\documentclass[12pt]{article}

\usepackage[margin=1in]{geometry} 
\usepackage{amsmath,amsthm,amssymb}
\usepackage{graphicx}
\usepackage{color}
\usepackage[dvipsnames]{xcolor}
\usepackage{subcaption}
\usepackage[export]{adjustbox}
\usepackage{algorithm}
\usepackage{algorithmicx}
\usepackage{lmodern}
\usepackage[noend]{algpseudocode}
\usepackage{tikz}
\usepackage[toc,page]{appendix}
\usepackage{multirow}
\usetikzlibrary{bayesnet}
\usepackage{subcaption}
\usepackage{authblk}

\makeatletter
\newcommand*{\rom}[1]{\expandafter\@slowromancap\romannumeral #1@}
\makeatother

\algnewcommand\algorithmicinput{\textbf{INPUT: }}
\algnewcommand\Input{\item[\algorithmicinput]}
\algnewcommand\algorithmicoutput{\textbf{OUTPUT: }}
\algnewcommand\Output{\item[\algorithmicoutput]}

\newcommand\sForAll[2]{ \ForAll{#1}#2\EndFor} 

\newtheorem{assumption}{Assumption}
\newtheorem{lemma}{Lemma}

\newtheorem{theorem}{Theorem}
\providecommand{\keywords}[1]
{
  \small	
  \textbf{\textit{Keywords---}} #1
}

\graphicspath{ {./pic/} }

\title{An adaptive Hessian approximated stochastic gradient MCMC method}
\author[1]{Yating Wang\thanks{E-mail: wang4190@purdue.edu}}
\author[1]{Wei Deng\thanks{E-mail: deng106@purdue.edu}}
\author[2]{Guang Lin\thanks{Corresponding Author, Fax: 765 494 0548; Tel: 765 494 1965; E-mail: guanglin@purdue.edu}}
\affil[1]{Department of Mathematics,   Purdue University, West Lafayette, IN 47907, USA}
\affil[2]{Department of Mathematics, School of Mechanical Engineering, Department of Statistics (Courtesy), Department of Earth, Atmospheric, and Planetary Sciences (Courtesy), Purdue University, West Lafayette, IN 47907, USA}

\begin{document}
	
\maketitle 
\begin{abstract}
	Bayesian approaches have been successfully integrated into training deep neural networks. One popular family is stochastic gradient Markov chain Monte Carlo methods (SG-MCMC), which have gained increasing interest due to their scalability to handle large datasets and the ability to avoid overfitting. Although standard SG-MCMC methods have shown great performance in a variety of problems, they may be inefficient when the random variables in the target posterior densities have scale differences or are highly correlated. In this work, we present an adaptive Hessian approximated stochastic gradient MCMC method to incorporate local geometric information while sampling from the posterior. The idea is to apply stochastic approximation to sequentially update a preconditioning matrix at each iteration. The preconditioner possesses second-order information and can guide the random walk of a sampler efficiently. Instead of computing and saving the full Hessian of the log posterior, we use limited memory of the sample and their stochastic gradients to approximate the inverse Hessian-vector multiplication in the updating formula. Moreover, by smoothly optimizing the preconditioning matrix, our proposed algorithm can asymptotically converge to the target distribution with a controllable bias under mild conditions. To reduce the training and testing computational burden, we adopt a magnitude-based weight pruning method to enforce the sparsity of the network. Our method is user-friendly and is scalable to standard SG-MCMC updating rules by implementing an additional preconditioner. The sparse approximation of inverse Hessian alleviates storage and computational complexities for large dimensional models. The bias introduced by stochastic approximation is controllable and can be analyzed theoretically.  Numerical experiments are performed on several problems, including sampling from 2D Gaussian distribution, regression problems, and learning the solutions of elliptic PDE. The numerical results demonstrate great improvement on both the convergence rate and accuracy.
	
\end{abstract}


\section{Introduction}
Deep learning has gained increasing interest in many areas due to its performance when dealing with large scale datasets. One important aspect of their successes in handling large datasets is that they process a small batch of data at each iteration to estimate the gradient of a cost function and update model parameters using gradient descent with a small step size. Bayesian approaches consider uncertainty in model parameters and help to improve the robustness in model learning. MCMC, as one of the most fashionable methods in Bayesian learning, is known for its asymptotic properties. However, it requires computations using the whole dataset, which is not feasible in large scale learning. 

In recent years, many efforts have been made to bring Bayesian methods into the learning of DNNs \cite{ahn2012bayesian, sgld, ding2014bayesian}. One of the most popular approaches is stochastic gradient Langevin dynamics (SGLD) \cite{sgld}. It is a stochastic gradient MCMC algorithm that originates from the discretization of Langevin diffusion. Similar to stochastic gradient descent (SGD), SGLD using mini-batches to approximate the gradients in the loss function. However, it injects a suitable amount of noise when updating parameters so that the sample variance matches the posterior variance. Moreover, with decreasing step sizes, it avoids the Metropolis-Hastings accept-reject step during sampling. It joins the stochastic optimization algorithm which resembles SGD, with Langevin dynamics which injects noise in the parameter updating formula. By injecting the right amount of noise, the method ensures that the trajectory of parameters will converge to the true posterior, rather than the MAP \cite{sgld, vollmer2016exploration, sg-mcmc-convergence}. 

However, due to the complexity of DNN architecture, the model parameters may have complicated posterior density functions \cite{dauphin2014identifying,PSGLD, chen2014stochastic}. When the parameters have different scales in different directions, it may be inefficient if adopting a common step size. It becomes even more sophisticated if the target densities are highly correlated. There have been a lot of methods in the optimization community to overcome these difficulties and accelerate the gradient descent, such as preconditioning and stochastic Newton-type method \cite{dauphin2015equilibrated, zhang2011quasi, byrd2016stochastic, bordes2009sgd}. However, directly applying these methods to SGLD will not produce a correct MCMC scheme \cite{PSGLD, simsekli2016stochastic, ma2015complete} in general. As indicated in \cite{xifara2014langevin, Riemann_LD_HMC, ma2015complete}, from another point of view, one can directly consider a Langevin diffusion on a Riemann manifold which described the geometric structure for the probability model. To ensure the diffusion has an invariant density, one needs to choose drift and volatility according to the Fokker–Planck equation, thus resulting in an additional drift term $\Gamma$. Several attempts have been made starting from the discretization of Riemann Langevin dynamics, to incorporate the underlying geometry according to the metric tensor in the sampling algorithm such that constant step size is adequate along with all directions.  These methods also replace the gradient of a cost function using estimation from mini-batches as in SGLD. For example, stochastic gradient Riemann Langevin dynamics (SGRLD) \cite{SGRLD} incorporates local curvature information by adopting the expected Fisher information as its metric tensor. However, the full second-order Fisher information is intractable to obtain in many applications. 

Preconditioned SGLD (PSGLD) is a computationally efficient method where a diagonal preconditioning matrix is employed as the metric tensor. In \cite{PSGLD}, the authors
adopt a diagonal preconditioner where it is updated sequentially taking into account the current gradient and preconditioning matrix in the previous time step. This type of preconditioner can handle scale differences in the target density but may not be sufficient for highly correlated densities. Moreover, the correction term $\Gamma$ needs computation of third-order derivatives, and ignoring the term in the updating equation will introduce a permanent bias on the MSE \cite{PSGLD}. To tackle these issues, a Hessian approximated stochastic gradient MCMC method (HAMCMC) \cite{simsekli2016stochastic} is studied, and it uses the local Hessian of the negative log posterior as an approximation to the full expected Fisher information. Instead of computing and storing the Hessian matrix, the limited memory BFGS (L-BFGS) algorithm \cite{liu1989limited, byrd2016stochastic} is employed to approximate the product of inverse Hessian and gradient vectors. The idea is to reduce the computation and storage burden while maintaining accuracy. In addition, the current parameter at time step $t$ is updated based on the sample at the previous time step $t-M$, and the approximated Hessian is computed using a history of samples at time steps $\{t-2M+1,\cdots, t-M+1,t-M-1,\cdots,t-1\}$. They claim that the correction term $\Gamma$ vanishes due to this construction. However, when $M$ is large, there will be a large gap between the two samples in the updating formula. Additionally, note that the memory size is $2M-2$ which is larger than the standard memory size $M$.

In this paper, we propose a stochastic Hessian approximated MCMC algorithm with the help of stochastic approximation (SA) to adaptively approximate the preconditioning matrix which involves the Hessian information. SA methods are typically used for root-finding problems or optimization problems in an iterative manner. It was first developed by Robbins and Monro \cite{robbins1951stochastic}, and serves as a typical framework in adaptive algorithms and control of stochastic systems. It naturally fits in our training of a Bayesian model and sequentially updates preconditioning matrices. Compared with HAMCMC, our proposed method (HASGLD-SA) requires fewer samples in the L-BFGS algorithm. We prove that the samples generated from the proposed algorithm weakly converge to the true posterior with a controllable bias introduced by stochastic approximation. The advantages of our proposed algorithm are (1) user-friendly: the implementation is more straightforward, the parameter at time step $t$ is updated based on the sample at the previous time step $t-1$ and there is no gap in the updating formula, (2) scalable: it requires less computation and memories, which is important in applications which require to run a very large-scale computational model, (3) the bias introduced by the algorithm is controllable and can be analyzed theoretically. Moreover, we adopt a magnitude-based weight pruning method to enforce the sparsity of the network, which further reduces the training and testing computational cost.

The plan of the paper is as follows. In Section \ref{sec:prelim}, we review some backgrounds in Langevin dynamics, Riemann Langevin dynamics, and some stochastic gradient MCMC algorithms. In Section \ref{sec:main_method}, our main algorithm is proposed. We first present a detailed online damped L-BFGS algorithm which is used to approximate the inverse Hessian-vector product and discuss the properties of the approximated inverse Hessian. Next, the adaptive Hessian approximated MCMC algorithm with the stochastic approximation to the preconditioning matrix is presented.
Its convergence is discussed in Section \ref{sec:analysis}. Applying the proposed method to a simple 2D Gaussian distribution, a large-p-small-n regression problem, and to solve elliptic problems with varying source terms or heterogeneous coefficients, we demonstrate the numerical examples in Section \ref{sec:numerical} and conclude in Section \ref{sec:conclusion}.

\section{Preliminary} \label{sec:prelim}
First, we present backgrounds on SGLD, preconditioned SGLD, and Hessian approximated SGLD.
\subsection{Langevin Dynamics and SGLD}
 Denote by $\boldsymbol \beta$ the model parameters in DNN. Let $D = \{d_i \}_{i=1}^N$ be the training dataset, where $d_i = (x_i, y_i)$ is an input-output pair. Let $p(\boldsymbol \beta)$ be a prior distribution, and $p(d|\boldsymbol \beta)$ be the likelihood function. The posterior distribution is then $p(\boldsymbol \beta|D) \propto p(\boldsymbol \beta) \prod_{i=1}^{N} p(d_i|\boldsymbol \beta)$. The stochastic differential equation (SDE) which yields an invariant distribution $p(\boldsymbol \beta|D)$ 
 \begin{equation}\label{eq:LD}
 d {\boldsymbol \beta}_t = \nabla_{\boldsymbol \beta} {L} (\boldsymbol {\beta}_{k} ) d t + \sqrt{2} d W_t
 \end{equation}
 where $W_t$ is a Brownian motion and
 \begin{equation*}
 \nabla_{\boldsymbol \beta} {L} (\boldsymbol {\beta}_{k} ) = \nabla_{\boldsymbol \beta} \log p(\boldsymbol \beta) +  \sum_{i=1}^N \nabla_{\boldsymbol \beta} \log p(d_{i}|\boldsymbol \beta)
 \end{equation*} 
 
 The likelihood for regression problem can be rewritten as
 \begin{equation*}
 p (d_k| \boldsymbol \beta, \sigma^2) = \frac{1}{(2\pi \sigma^2)^{n/2}} \exp \big\{-\frac{  \sum \limits_{x^k_i \in d_k} (x^k_i -\mathcal{F}(x^k_i; \boldsymbol \beta)  )^2 }{ 2 \sigma^2 } \big\}
 \end{equation*}
 where $\mathcal{F}$ denotes a model describing the input-output map between $x_i^k$ and $y_i^k$.
 
SGLD is a posterior Bayesian sampling method originates from the discretization of the SDE \eqref{eq:LD} and combines the idea from stochastic gradient algorithms. The loss gradient can be approximated efficiently using mini-batches, and the uncertainty in the model parameter can be captured in Bayesian learning. It avoids the MH correction by adopting small learning rate. The model parameters update as follows:
\begin{equation*}
{\boldsymbol \beta}_{k+1} = {\boldsymbol \beta}_{k} + \epsilon_k \nabla_{\boldsymbol \beta} \tilde{L} ({\boldsymbol \beta}_{k} ) + \mathcal{N}(0, 2\epsilon_k \tau^{-1})
\end{equation*}
where $\epsilon_k$ is the learning rate and 
\begin{equation*}
\nabla_{\boldsymbol \beta} \tilde{L} (\boldsymbol {\beta}_{k} ) = \nabla_{\boldsymbol \beta} \log p(\boldsymbol \beta) + \frac{N}{n} \sum_{i=1}^n \nabla_{\boldsymbol \beta} \log p(d_{ki}|\boldsymbol \beta)
\end{equation*} 
is the stochastic gradient computed from a mini-batch $d_k = \{d_{k1}, \cdots, d_{kn} \}$.

\subsection{Reimann Langevin Dynamics and PSGLD, HASGLD}
Stochastic Gradient Riemann Langevin Dynamics (SGRLD) \cite{SGRLD} is a generalization of SGLD on a Riemannian manifold. If the components of the model parameter $\boldsymbol \beta$ possess different scales or are highly correlated, the invariant probability distribution for the Langevin equation is not isotropic, using standard Euclidian distance may lead to slow mixing. Given with some metric tensor $G^{-1}(\boldsymbol \beta)$, the SDE defining the Langevin diffusion with stationary distribution $p(\boldsymbol \beta|D)$ on a Riemann manifold is
 \begin{equation}\label{eq:RLD}
d {\boldsymbol \beta}_t = \left[ G(\boldsymbol \beta) \nabla_{\boldsymbol \beta} {L} (\boldsymbol {\beta} ) + \Gamma(\boldsymbol \beta) \right] d t + \sqrt{2G(\boldsymbol \beta) } d W_t
\end{equation}
where $\Gamma_i({\boldsymbol \beta}) = \sum_j \frac{\partial G_{ij}(\boldsymbol \beta)  }{\partial \beta_j}$. We note that $\Gamma({\boldsymbol \beta})$ corresponds to variations in local curvature on the manifold and is equal to zero for a constant curvature. It is shown that, the invariant distribution of the dynamics \eqref{eq:RLD} is $p^s(\beta) \propto \exp{{L} (\boldsymbol {\beta} )}$, and it is unique if $G^{-1}(\boldsymbol \beta)$ is positive definite \cite{ma2015complete}. 

In this case, the parameter updates can be guided using the geometric information of this manifold:
\begin{equation} \label{eq:sgrld}
{\boldsymbol \beta}_{k+1} = {\boldsymbol \beta}_{k} + \epsilon_k  \left[ G(\boldsymbol \beta_k ) \nabla_{\boldsymbol \beta} \tilde{L} ({\boldsymbol \beta}_{k} ) + \Gamma({\boldsymbol \beta}_k)  \right] +  \sqrt{ 2\epsilon_k \tau^{-1}G(\boldsymbol \beta_k )  }z_k
\end{equation}
where $z_k \sim \mathcal{N}(0, I)$.

A natural choice for metric tensor is the expected Fisher information matrix, however, it is intractable in many cases. In \cite{PSGLD}, the authors introduce a diagonal preconditioner, which resembles the preconditioning matrix in RMsProp, to reduce computational cost. However, it is effective to handle the case when there are scale differences among model parameters, but may not be sufficient to deal with strongly correlated target densities.
A Hessian-approximated MCMC \cite{simsekli2016stochastic} method (HAMCMC) was proposed to overcome this issue. The idea is to compute the local curvature of the target density by approximating local Hessian information via quasi-newton approaches. In particular, HAMCMC generates samples $\boldsymbol \beta_k$ based on $\boldsymbol \beta_{k-M}$, where $M \geq 2$, and uses a history of samples $\{\boldsymbol \beta_{k-2M+1}, \cdots,  \boldsymbol \beta_{k-M-1}, \boldsymbol \beta_{k-M+1}, \cdots, \boldsymbol \beta_{k-1} \}$ to approximate inverse Hessian information via limited BFGS. By this construction, the authors claim that the approximated Hessian is independent of the base-line sample $\boldsymbol \beta_{k-M}$, thus the correction term $\Gamma({\boldsymbol \beta}_k)$ can be ignored without introducing additional bias. However, if the memory size $M$ is large, there will be a large gap between two neighboring samples in the update rule. This may require a larger regularizer to ensure positive definite L-BFGS approximations, which result in a preconditioning matrix close to the identity matrix.

In this work, we adopt the stochastic approximation (SA) idea to iteratively update the approximated inverse Hessian. In each step, we sample $\boldsymbol \beta_k$ based on $\boldsymbol \beta_{k-1}$, and approximate $G(\boldsymbol \beta_k)$ using history samples $\{\boldsymbol \beta_{k-M+1}, \cdots,  \boldsymbol \beta_{k-1} \}$. Compared with HAMCMC, our proposed method (HAMCMC-SA) requires fewer samples in the memory.

\section{Main Method} \label{sec:main_method}

\subsection{The online damped L-BFGS algorithm } 

Now, we describe the online damped L-BFGS algorithm to approximate the local inverse Hessian at each iteration. In this approach, the approximated inverse Hessian matrix does not need to be computed or stored explicitly, but an approximation to the matrix-vector product is updated using successive gradient vectors instead.  

Suppose we have a history of samples $\{\boldsymbol \beta_{k-M+1}, \cdots,  \boldsymbol \beta_{k-1} \}$, where $M$ is the memory size. Let $s_k = \boldsymbol \beta_{k+1} - \boldsymbol \beta_k$ be the increment in samples, and $y_k = \nabla_{\boldsymbol \beta} \tilde{L} ({\boldsymbol \beta}_{k+1}, d_k ) -  \nabla_{\boldsymbol \beta} \tilde{L} ({\boldsymbol \beta}_k,d_k)$ be the differences between sample gradients.  We remark that, here the stochastic gradients $\nabla_{\boldsymbol \beta} \tilde{L} ({\boldsymbol \beta}_{k+1}, d_k)$ and $\nabla_{\boldsymbol \beta}\tilde{L} ({\boldsymbol \beta}_k,d_k)$ are evaluated with respect to the same set of samples $d_k$, which refers to the online L-BFGS \cite{mokhtari2015global}. This will avoid additional differences between noisy gradient estimates and will only be applied for determining the stochastic gradient variation.

Another thing to mention is that, given an initial guess of Hessian approximation which is positive definite, the curvature condition $s_k^T y_k >0$ needs to be satisfied such that after $p$ recursion steps the Hessian approximation is still positive definite. The following techniques will be adopted, the curvature condition is guaranteed by Lemma \ref{lemma:curvature}.

\begin{equation}\label{eq:ybar}
\bar{y}_k = \theta_k y_k +   (1-\theta_k)s_k
\end{equation} 
where 
\begin{equation*}
\theta_k = \begin{cases}
\displaystyle{ \frac{(1-r ) s_k^T s_k}{s_k^T B_{k,0} s_k - s_k^T B_{k,0}y_k} }, & \text{if} \displaystyle{  \ s_k^T y_k < r s_k^T B_{k,0} s_k} \\
1, & \text{otherwise}
\end{cases} 
\end{equation*} 
where $0<r<1$ is a constant, $B_{k,0}$ is the initial guess of the Hessian at $k$-th step.

The online damped L-BFGS approximation of Hessian employs the update formula:
\begin{equation}\label{eq:B_update}
B_{k,i+1} = B_{k,i} + \frac{\bar{y}_j \bar{y}_j^T }{\bar{y}_j^T s_j} - \frac{ B_{k,i} s_j s_j^T B_{k,i} }{ s_j^T B_{k,i}s_j } 
\end{equation}
where $j = k-M+i$, $M$ denotes the memory size. The initial guess of the recursion is typically chosen to be $B_{k,0} = \gamma_k I$, where $ \displaystyle{ \gamma_k = \text{max}\{\frac{\bar{y}_k \bar{y}_k^T}{s_k \bar{y}_k^T}, \delta\} }$. Denote by $B_k$ be the final approximation of the Hessian, and $\tilde{G}_k = B_k^{-1}$. After $M$ recursions, we take $B_{k}=B_{k,M}$.

For the inverse Hessian, we have
\begin{equation}\label{eq:G_update}
\tilde{G}_{k,i+1} = (I - \frac{s_j \bar{y}_j ^T}{ \bar{y}_j ^T s_j}) \tilde{G}_{k,i}   (I - \frac{s_j \bar{y}_j ^T}{ \bar{y}_j^T s_j})^T + \frac{s_j s_j^T }{\bar{y}_j^T s_j} 
\end{equation}
The initial guess of the recursion is $\tilde{G}_{k,0} = \gamma_k^{-1} I$.

As for the $\sqrt{ \tilde{G}_k }:= S_k$ or $\tilde{G}_k =  S_k S_k^T$,
\begin{align*}
&S_{k,i+1} = (I - p_j q_j^T) S_{k,i}\\
& p_j = \frac{s_j}{s_j^T \bar{y}_j}, \;\;\;\;\;\;\; q_j = \sqrt{\frac{s_j^T \bar{y}_j}{s_j^T B_{k,i} s_j}}B_{k,i} s_j - \bar{y}_j
\end{align*}

For brevity, denote by $g_k = \nabla_{\boldsymbol \beta} \tilde{L} ({\boldsymbol \beta} )$. The online damped L-BFGS algorithm to compute $G_k g_k$ and $\sqrt{G_k }z_k$ use a two-loop recursion and is described in Algorithm \ref{alg:l-bfgs}.
\begin{algorithm} [!htb]
	\caption{L-BFGS}\label{alg:l-bfgs}
	\begin{algorithmic}[1] 
		\Input{Initialize $g_k$, $z_k$,  $M$, $B_{k,0} = \gamma_k I$, $S_{k,0}  = 1/\sqrt{\gamma_k} I$, $\tilde{G}_{k,0} = \gamma_k^{-1} I$ }
		\Output {$\xi = \tilde{G}_{k,M} g_k$, $\eta = S_{k,M} z_k$}
		\State $q \gets g_k$
		\sForAll{ $i \gets k: \text{min}\{k-M+1, 0\}$} {
			\State $\displaystyle{ \alpha_i \gets \frac{s_i^T q}{\bar{y}_i^T s_i} }$
			\State $q \gets q - \alpha_i \bar{y}_i$}

		\State $a_1 \gets B_0 s_{k-M+1}$, $T_{1,j} = B_0 s_{k-M+j}, \,\, j = 1, \cdots, M$
		
		\sForAll{ $i \gets 2: k$} {
			\sForAll{ $j \gets i : k$}{
				\State $\displaystyle{T_{i,j} \gets T_{i-1,j} + \frac{\bar{y}_{i-1}^T s_j}{s_{i-1}^T \bar{y}_{i-1}} y_{i-1} - \frac{a_{i-1}^T s_j}{s_{i-1}^T a_{i-1}} a_{i-1} }$	}
			\State $a_i \gets T_{i,i}$}
		
		\State $\xi \gets  \tilde{G}_{k,0} q$
		\State $\eta \gets S_{k,0} z_k$
		\sForAll{ $i \gets \text{min}\{k-M+1, 0\} : k$} {
			\State $\displaystyle{\beta_i \gets \frac{\bar{y}_i^T p}{\bar{y}_i^T s_i}}$
			\State $\xi \gets \xi + (\alpha_i -\beta_i) s_i$
			\State  $\displaystyle{ \eta \gets \eta - \frac{a_i^T \eta}{\sqrt{s_i^T \bar{y}_i} \sqrt{a_i^T s_i} } s_i - \frac{\bar{y}_i^T \eta }{s_i^T \bar{y}_i} s_i}$}	
	\end{algorithmic}
\end{algorithm}

 \begin{lemma}\label{lemma:curvature}
Let $\bar{y}_j$ be defined in \eqref{eq:ybar}, if $B_{k,i}$ and $\tilde{G}_{k,i}$ are positive definite, then $B_{k,i+1}$ and $\tilde{G}_{k,i+1}$ generated by \eqref{eq:B_update} and \eqref{eq:G_update} are both positive definite. 
\end{lemma}
\begin{proof}
	By \eqref{eq:ybar}, we can easily obtain
	\begin{equation*}
	s_j^T \bar{y}_j = 	\begin{cases}
	\displaystyle{ r s_j^T B_{k,0} s_j }, & \text{if } \displaystyle{  s_j^T y_j < r s_j^T B_{k,0} s_j} \\
	s_j^T y_j, & \text{otherwise}
	\end{cases} 
	\end{equation*}
	
	 Thus,  $s_j^T \bar{y}_j \geq  \lambda s_j^T B_{k,0} s_j >0$ since $B_{k,i}$ is positive definite. By positive definiteness of $\tilde{G}_{k,i}$, for $\textbf{x}$, we have
	 \begin{equation*}
	  \textbf{x}^T\tilde{G}_{k,i} \textbf{x} > 0
	 \end{equation*}
	 Then it's easy to see that $s_j^T \bar{y}_j>0$, and
	  \begin{align*}
	 \textbf{x}^T \tilde{G}_{k,i+1}\textbf{x} &=  \textbf{x}^T (I - \frac{s_j \bar{y}_j ^T}{ \bar{y}_j ^T s_j}) \tilde{G}_{k,i}   (I - \frac{s_j \bar{y}_j ^T}{ \bar{y}_j^T s_j})^T  \textbf{x}+ \frac{1}{\bar{y}_j^T s_j}  \textbf{x}^Ts_j s_j^T  \textbf{x}\\
	 &= \textbf{z}^T  \tilde{G}_{k,i} \textbf{z} +  \frac{1}{\bar{y}_j^T s_j} || s_j^T  \textbf{x}||^2 > 0
	 \end{align*}
	 where $ \displaystyle{\textbf{z} = (I - \frac{s_j \bar{y}_j ^T}{ \bar{y}_j^T s_j})^T}\textbf{x}$.
	 Thus, $\tilde{G}_{k,i+1}$ is positive definite, so is $B_{k,i+1}$.
	 
\end{proof}

 \begin{assumption}
 The eigenvalues of the Hessian $H_k = \nabla^2 \tilde{L}(\beta_k)$ are bounded between constants $0<a$ and $A<\infty$, i.e,
 \begin{equation*}
 a \preceq H_k \preceq A
 \end{equation*}
 \end{assumption}

\begin{lemma}  \label{lemma:eigen_bd}
	The eigenvalues of Hessian approximation $B_k$ generated from iteration \eqref{eq:B_update} with $B_{k,0} = \gamma_k I$ are uniformly bounded,
	 \begin{equation*}
	\tilde{a} \preceq B_k \preceq \tilde{A}
	\end{equation*}
	 
\end{lemma}

\begin{proof}
	Take trace of the matrix in both hands of the equation \eqref{eq:B_update}, we have
	\begin{equation} \label{eq:trace1}
	\text{tr}( B_{k,i+1} ) = 	\text{tr}(B_{k,i}) + \frac{1}{\bar{y}_j^T s_j} \text{tr}(\bar{y}_j \bar{y}_j^T)- \frac{1}{ s_j^T B_{k,i}s_j }\text{tr}( B_{k,i} s_j s_j^T B_{k,i} )
	\end{equation}
	By the properties of trace of a matrix, the above equation can be simplified as
	\begin{align*}
	\text{tr}( B_{k,i+1} ) = 	\text{tr}(B_{k,i}) + \frac{\bar{y}_j^T \bar{y}_j}{\bar{y}_j^T s_j} - \frac{||B_{k,i} s_j||^2 }{ s_j^T B_{k,i}s_j } \leq \text{tr}(B_{k,i}) + \frac{\bar{y}_j^T \bar{y}_j}{\bar{y}_j^T s_j} 
	\end{align*}
	since $\displaystyle{ \frac{||B_{k,i} s_j||^2 }{ s_j^T B_{k,i}s_j } > 0}$ by the positive definiteness of $B_{k,i}$.
	
	Now we derive a bound for $\displaystyle{ \frac{\bar{y}_j^T \bar{y}_j}{\bar{y}_j^T s_j} }$. Since $B_{k,0} = \gamma_k I$, 
		\begin{equation}\label{eq:trace2}
	\frac{\bar{y}_j^T \bar{y}_j}{\bar{y}_j^T s_j} = \frac{ ||\theta_j y_j + (1-\theta_j)B_{k,0} s_j||^2 }{r  s_j^TB_{k,0} s_j}  = \frac{1}{r} \left( \frac{\theta_j^2 ||y_j||^2 }{ \gamma_k ||s_j||^2}  + (1-\theta_j)^2 \gamma_k + \frac{2\theta_j(1-\theta_j)  y_j^T s_j }{ ||s_j||^2} \right)
	\end{equation}
	Denote by $\bar{H} = \int_{0}^{1} H(\beta_k + \tau (\beta_{k+1}-\beta_k)) d\tau$ the mean of Hessian in the segment $[\beta_k, \beta_{k+1} ]$, $ a \preceq \bar{H} \preceq A$. 
	Due to the fact
	\begin{equation} \label{eq:calc1}
	\frac{\partial \nabla \tilde{L} (\beta_k + \tau (\beta_{k+1}-\beta_k)) } {\partial \tau} = (\beta_{k+1}-\beta_k) H(\beta_k + \tau (\beta_{k+1}-\beta_k)), 
	\end{equation}
	we have 
		\begin{equation}\label{eq:calc2}
	\int_{0}^{1} (\beta_{k+1}-\beta_k) H(\beta_k + \tau (\beta_{k+1}-\beta_k))  d\tau  =\nabla  \tilde{L} (\beta_{k+1}) -\nabla \tilde{L} (\beta_{k}),
	\end{equation}
	by integrating \eqref{eq:calc1} over $[0,1]$ in both sides of the equation. 
	One can easily see from \eqref{eq:calc2} that $\bar{H} s_k = y_k$. Thus, the first term and third term in \eqref{eq:trace2} can be bounded as follows
			\begin{align*}
	\frac{\theta_j^2 ||y_j||^2 }{\gamma_k||s_j||^2} & \leq 	\frac{ ||\bar{H} s_j ||^2 }{\gamma_k||s_j||^2} \leq \frac{A^2}{\delta} \\
	\frac{2\theta_j(1-\theta_j)  y_j^T s_j }{ ||s_j||^2}  & \leq 2	\frac{ s_j^T \bar{H} s_j  }{||s_j||^2} \leq 2 A
	\end{align*}
	since $0<\theta_j<1$, and $ \delta \leq \gamma_k \leq \delta + A$.
	Plug these estimates in \eqref{eq:trace2}, we get
		\begin{equation*}
	\frac{\bar{y}_j^T \bar{y}_j}{\bar{y}_j^T s_j}   \leq \frac{1}{r} \left( \frac{A^2}{\delta} + (\delta + A) + 2 A \right) 
	\end{equation*}
	Then \eqref{eq:trace1} can be bounded as
		\begin{align*}
	\text{tr}( B_{k,i+1} ) & \leq	\text{tr}(B_{k,i}) + \frac{1}{r} \left( \frac{A^2}{\delta} + (\delta + A) + 2 A \right) \\
	&\leq 	\text{tr}(B_{k,0}) + \frac{M}{r} \left( \frac{A^2}{\delta} + (\delta + A) + 2 A \right) \\
	&\leq 	d (\delta + A) + \frac{M}{r} \left( \frac{A^2}{\delta} + (\delta + A) + 2 A \right)   
	\end{align*}
	where $d$ is the size of matrix $B_{k,0}$, $M$ is the number of recursions in BFGS update.
	
	Since $B_{k,i+1}$ is positive definite, and $\text{tr}( B_{k,i+1} )$ is the sum of all eigenvalues of $B_{k,i+1}$, the largest eigenvalue $\mu_{\text{max}}$ of $B_{k,i+1}$ satisfies
	\begin{equation*}
	\mu_{\text{max}} \leq d (\delta + A) + \frac{M}{\lambda} \left( \frac{A^2}{\delta} + (\delta + A) + 2 A \right) := \tilde{A} 
	\end{equation*}

	Thus the largest eigenvalue of  $B_{k,i+1}$ is no greater than $\tilde{A}$.
	
	On the other hand, 
	\begin{equation} \label{eq:det}
	\text{det} ( B_{k,i+1}) = \text{det} ( B_{k,i})  \text{det}  \left( I +   \frac{  B_{k,i}^{-1} \bar{y}_j \bar{y}_j^T }{\bar{y}_j^T s_j} - \frac{ s_j ( B_{k,i} s_j)^T}{ s_j^T B_{k,i}s_j }     \right) 
	 \end{equation}
	 The second term in the right hand side of \eqref{eq:det} is equivalent to
	 \begin{equation*}
	 \text{det}  \left( I +   \frac{  B_{k,i}^{-1} \bar{y}_j \bar{y}_j^T }{\bar{y}_j^T s_j} - \frac{ s_j ( B_{k,i} s_j)^T}{ s_j^T B_{k,i}s_j }     \right)  = (1+u_1^T u_2)(1+u_3^T u_4) - (u_1^T u_4)(u_2^T u_3)
	 	\end{equation*}
	 	where $\displaystyle{ u_1 = -s_j, \;\; u_2 = \frac{ B_{k,i}s_j}{s_j^T B_{k,i}s_j }}, \;\; u_3 = B_{k,i}^{-1} \bar{y}_j,  \;\; u_4=  \frac{\bar{y}_j }{\bar{y}_j^T s_j}$.
	 	
	 	It is easy to check that $u_1^T u_2 = -1$, $u_1^T u_4 = -1$, $u_2^T u_3 = \frac{ s_j^T \bar{y}_j}{s_j^T B_{k,i}s_j }$, thus \eqref{eq:det} implies
	 		\begin{equation*} 
	 	\text{det} ( B_{k,i+1}) = \text{det} ( B_{k,i})  \frac{ s_j^T \bar{y}_j}{s_j^T B_{k,i}s_j } \geq   \text{det} ( B_{k,i})  \frac{ r \gamma_k ||s_j||^2  }{\tilde{A} ||s_j||^2 } =  \text{det} ( B_{k,i})  \frac{ r \gamma_k}{\tilde{A} } 
	 	\end{equation*}
	 since $s_j^T \bar{y}_j \geq  r s_j^T B_{k,0} s_j \geq r \gamma_k ||s_j||^2$ and $s_j^T B_{k,i}s_j \leq A ||s_j||^2$.
	 
	 	By induction and using the fact that $\text{det} ( B_{k,0}) = \gamma_k^d$, we have
	 	\begin{equation*} 
	 	\text{det} ( B_{k,i+1}) \geq  \text{det} ( B_{k,0}) \left( \frac{ r \gamma_k}{\tilde{A} } \right) ^M
\geq \left( \frac{ r }{\tilde{A} } \right) ^M\gamma_k^{d+M} \geq \left( \frac{ r }{ \tilde{A} } \right) ^M \delta^{d+M} 
	\end{equation*}
	 Since any eigenvalue of $B_{k,i+1}$ is no greater than $\tilde{A}$, and $	\text{det} ( B_{k,i+1})$ is equal to the product of all eigenvalues, we have that for any specific eigenvalue $\mu_j$ of $B_{k,i+1}$ 
		 	\begin{equation*} 
	\mu_j \geq  \frac{1}{\tilde{A}^{d-1}} \left( \frac{ r }{\tilde{A} } \right) ^M \delta^{d+M} 	:= \tilde{a}
	\end{equation*}	
Thus, we have
 \begin{equation*}
\tilde{a} \preceq B_k \preceq \tilde{A}.
\end{equation*}
	Furthermore,
	 \begin{equation*}
 \frac{1}{\tilde{A}} \preceq \tilde{G}_k \preceq \frac{1}{\tilde{a}}.
	\end{equation*}
	
\end{proof}

\subsection{Adaptive Hessian-approximated SG-MCMC with iterative pruning}
The adaptive Hessian-approximated stochastic gradient MCMC with iterative pruning is a mixture of optimization and sample algorithm, where the model parameters are sampled from \eqref{eq:sgrld}, and the preconditioning matrix $G(\boldsymbol \beta)$ is optimized iteratively.

The idea is to obtain the optimal $G_*$ based on the asymptotically correct distribution $\pi(\boldsymbol{\beta})$ through stochastic approximation. We aim to get an estimate $G_*$ which solves the fixed point equation $\displaystyle{ \int g_{G}(\boldsymbol{\beta}) \pi(\boldsymbol{\beta}) d \boldsymbol{\beta} = G_*}$, where $g_{G}(\cdot)$ denotes some mapping to derive the optimal $G$ given current $\boldsymbol{\beta}$.

Define the random output $H(\boldsymbol{\beta}, G) = g_{G}(\boldsymbol{\beta})   - G$ and its mean field function $h(G) = \mathbb{E} [H(\boldsymbol{\beta}, G)]$.  In our approach, we approximate $g_{G}(\boldsymbol{\beta})$ using the damped online L-BFGS as described in Algorithm \ref{alg:l-bfgs}. This will result a bias $\delta(M,n,\epsilon_k)$ at each step which includes the error introduced by using stochastic gradients, and the error introduced by using a limited memory instead of full memory. Here $M$ is the memory size, $n$ is the number of samples in a mini-batch. That is, we use 
\begin{equation}\label{eq:H_tilde}
\tilde{H}(\boldsymbol{\beta}, G) =H(\boldsymbol{\beta}, G)  + \delta(M,n,\epsilon_k),
\end{equation}
where we assume $ \mathbb{E}|| \delta(M,n,\epsilon_k)||^2 \leq C_0^2$.

After sampling $\boldsymbol{\beta}_{k+1}$ using  \eqref{eq:sgrld} with approximated preconditioning matrix $G_k$, one can then update $G_{k+1}$ from the following recursion:
\begin{equation}\label{eq:G_sa}
G_{k+1} = G_k + \omega_{k+1} \tilde{H}(\boldsymbol{\beta}_{k+1}, G_k).
\end{equation}

In summary, the adaptive empirical Bayesian algorithm samples $\boldsymbol \beta$ and optimize $G(\boldsymbol{\beta})$ as in Algorithm \ref{alg:hSGLD-SA}.
\begin{algorithm} [!htb]
	\caption{AHAMCMC-SA}\label{alg:hSGLD-SA}
	\begin{algorithmic}[1] 
		\Input{Initialize $\boldsymbol{\beta}_1$, $M$, $p$, $G_1 = I$}
		\sForAll{ $k \gets 1: \#iterations$} {
			\State $g(\boldsymbol{\beta}_{k}) \gets \nabla_{\boldsymbol \beta} \tilde{L} (\cdot|d_k )$
			\State $z_k \sim \mathcal{N}(0, I)$
			\State $ \tilde{G}_k g_k \gets \xi$, $\tilde{S}_k z_k \gets \eta$ from  L-BFGS algorithm \ref{alg:l-bfgs}
			\State $\displaystyle{ G_k g_k \gets (1-\omega_{k+1}) G_{k-1} g_k  + \omega_{k+1} \tilde{G}_k g_k }$
			\State $\displaystyle { S_k z_k \gets (1-\omega_{k+1}) S_{k-1} z_k + \omega_{k+1} \tilde{S}_k z_k}$
			\State $\xi_k \gets G_k g_k / ||G_k g_k || $
			\State $\eta_k \gets S_k z_k/|| S_k z_k ||  $
			\State $\displaystyle{\boldsymbol{\beta}_{k+1} \gets \boldsymbol{\beta}_{k} + \epsilon_k \xi_k  + \sqrt{ 2\epsilon_k \tau^{-1}} \eta_k}$

			\If {Pruning}
			\State Prune the bottom -$p\%$ weights with lowest magnitude
			\State Increase the sparse rate	
			\EndIf
	}
	\end{algorithmic}
\end{algorithm}

\section{Convergence analysis}\label{sec:analysis}
In this section, we will discuss the convergence of stochastic approximation and the proposed algorithm. 
\subsection{Convergence of stochastic approximation of preconditioning matrix}
Denote by $\vec{G}$ vectorization of a matrix $G$, we first state the following stability lemma.
\begin{lemma} \label{lemma:stability}
	The mean field function $h(G)$ satisfies $\forall G \in  \Theta $, $\langle \vec{h(G)}, \vec{G} - \vec{G_*} \rangle \leq -||\vec{G} - \vec{G_*}||^2$, where $||\cdot||$ denotes $l_2$ norm. The mean field system $\frac{d \vec{G}}{d t} = \vec{h(G)}$ is globally asymptotically stable and $G_*$ is the globally asymptotically stable equilibrium.
\end{lemma}

\begin{proof}
	Since  $H(\boldsymbol{\beta}, G) = g_{G}(\boldsymbol{\beta})  - G$, the mean field function $h(G)$ is 
	\[
	h(G) = \int \left( g_{G}(\boldsymbol{\beta}) -G \right) \pi(\boldsymbol{\beta}) d \boldsymbol{\beta} = G_* - G
	\]
   Then,
	\[
	\langle \vec{h(G)}, \vec{G} - \vec{G_*} \rangle = -||\vec{G} - \vec{G_*}||^2 \leq -||\vec{G} - \vec{G_*}||^2
	\]
		Consider the positive definite Lyapunov function $V(\vec{G}) = \frac{1}{2} ||\vec{G_*} - \vec{G}||^2$, it's easy to see that $\langle \nabla V, \frac{d \vec{G}}{d t} \rangle  = \langle \vec{G} - \vec{G_*} ,  \vec{G_*} - \vec{G} \rangle = -||\vec{G} - \vec{G_*}||^2<0$, which completes the proof.
	
\end{proof}

\begin{assumption}
	The step size $\{\omega_k\}$ satisfies
	\begin{align*}
	\sum_{k=1}^{\infty} \omega_k  = +\infty, \;\;\;\;\;\;\; \sum_{k=1}^{\infty} \omega_k^2  < +\infty \\
	{\lim \inf}_{k\rightarrow \infty} 2\frac{\omega_k}{\omega_{k+1}} + \frac{\omega_{k+1}-\omega_k}{\omega_{k+1}^2} > 0 
	\end{align*}
	
\end{assumption}
In practice, one can choose $\omega_{k} = c_1(k+c_2)^{-\alpha}$ for $\alpha \in (0,1])$ and constants $c_1, c_2$.

%
%
%

\begin{lemma}  \label{lemma:unifrom_G}
	There exists $Q>0$, such that $\sup \mathbb{E}||G_k||^2 \leq Q^2$.
\end{lemma}
\begin{proof}
	From Lemma \ref{lemma:eigen_bd}, we have
	 \begin{equation*}
	 \frac{1}{\tilde{A}} \preceq \tilde{G}_k \preceq \frac{1}{\tilde{a}}
	\end{equation*}
	We will prove by induction. For $k=0$, $\mathbb{E}||G_0||^2 \leq  \frac{1}{\tilde{a}} := Q$. Assume we have $\mathbb{E}||G_k||^2 \leq  Q$, then
	\begin{align*}
\mathbb{E}||G_{k+1}||^2 &= 	\mathbb{E}||(1-\omega_{k})G_{k} +\omega_{k} \tilde{G}_{k+1}||^2 \\
&\leq (1-\omega_{k})^2\mathbb{E}||G_{k}||^2 + 2 (1-\omega_{k})\omega_{k} \sqrt{ \mathbb{E}||G_{k}||^2   \mathbb{E}||\tilde{G}_{k+1}||^2 }   + \omega_{k}^2 \mathbb{E}|| \tilde{G}_{k+1}||^2\\
&\leq (1-\omega_{k})^2 Q^2 + 2 (1-\omega_{k})\omega_{k} \sqrt{Q^2  (\frac{1}{\tilde{a}})^2 } + \omega_{k}^2 (\frac{1}{\tilde{a}})^2 \leq Q^2.
	\end{align*}
	This completes the proof. 
\end{proof}

\begin{assumption} \label{assumption:poisson}
	For all $G\in \Theta$, there exists a function $\mu_G(\boldsymbol \beta)$ that solves the Poisson equation 
	$\mu_G(\boldsymbol \beta) - \Pi_G \mu_G(\boldsymbol \beta) = H(G,\boldsymbol \beta) - h(G)$. 
	There exists a constant $C$ such that
	\begin{align*}
	\mathbb{E} ||  \Pi_G \mu_G(\boldsymbol \beta) || &\leq C\\
	\mathbb{E} ||  \Pi_G \mu_G(\boldsymbol \beta) - \Pi_{G'} \mu_{G'}(\boldsymbol \beta)||  &\leq C ||G-G'||
	\end{align*}
	Here $||\cdot||$ denote the Frobenius norm.
\end{assumption}

\begin{lemma}
	There exists a constant $Q_2>0$ such that
	\begin{equation}
		|| \tilde{H}(\boldsymbol{\beta}, G) ||^2 \leq Q_2(1+||G_k-G_*||^2)
	\end{equation}

\end{lemma}

\begin{proof}
	\begin{equation*}
||H(\boldsymbol{\beta}, G) ||^2 \leq  2||g_G(\boldsymbol{\beta})||^2 + 2 ||G_k||^2 \leq 2(\frac{1}{\tilde{a}})^2  + 2 ||G_k||^2 \leq C_1 (1+||G_k||^2 ) \leq \tilde{C}_1 (1+||G_k-G_*||^2).
	\end{equation*}
Then		
			\begin{align*}
		|| \tilde{H}(\boldsymbol{\beta}, G) ||^2 &= ||H(\boldsymbol{\beta}, G)  + \delta(M,n,\epsilon_k)||^2 \leq 2||H(\boldsymbol{\beta}, G) ||^2+2|| \delta(M,n,\epsilon_k)||^2\\
		&\leq 2 \tilde{C}_1 (1+||G_k-G_*||^2) + 2C_0^2 \leq Q_2 (1+||G_k-G_*||^2) 
		\end{align*}
		where $Q_2= 2 \tilde{C}_1 +  2C_0^2$.
\end{proof}

\begin{lemma} \label{lemma:Lamda}
	Let $k_0$ be an integer which satisfies 
	\begin{equation*}
	\inf_{k\geq k_0} \frac{\omega_{k+1} - \omega_k}{\omega_k\omega_{k+1} } + 2 -Q\omega_{k+1}  > 0
	\end{equation*}
	Then $\forall k \geq k_0$, the sequence $\{ \Lambda_k^K\}_{k=k_0}^{K}$ is increasing, where
	\begin{equation*}
	 \Lambda_k^K = 	\begin{cases}
	 \displaystyle{ 2\omega_k \prod_{j=k}^{K-1} (1-2\omega_{k+1} + Q\omega_{k+1}^2)}, & \text{if } \displaystyle{  k < K} \\
	 2\omega_k, & \text{if } \displaystyle{  k \geq K} 
	 \end{cases} 
	\end{equation*}
\end{lemma}

\begin{lemma}\label{lemma:sequence}
	There exists $\lambda_0$ and $k_0$ such that $\forall \lambda \geq\lambda_0$ and $\forall k \geq k_0$, the sequence $\displaystyle{ \{\psi_k \}_{k=1}^{\infty} }$  with $\displaystyle{ \psi_k = \lambda \omega_k + 2Q \sup_{i\geq k_0}  \triangle_i}$ satisfies
	\begin{equation}\label{eq:psi}
	\psi_{k+1} \geq (1-2\omega_{k+1} +Q\omega_{k+1}^2 )	\psi_k+ 14CQ\omega_{k+1}^2 + 4Q\triangle_k  \omega_{k+1}   
	\end{equation}
\end{lemma}

\begin{proof}
	Plug in $\displaystyle{ \psi_k = \lambda \omega_k + 2Q \sup_{i\geq k_0}  \triangle_i}$ in equation \eqref{eq:psi}, it's equivalent to 
		\begin{equation*}
	(\lambda \omega_{k+1} + 2Q \sup_{i\geq k_0}  \triangle_i) \geq (1-2\omega_{k+1} +Q\omega_{k+1}^2 )	(\lambda \omega_k + 2Q \sup_{i\geq k_0}  \triangle_i)+ 14CQ\omega_{k+1}^2 + 4Q\triangle_k  \omega_{k+1} 
	\end{equation*}
	Rearranging terms, we need to show
			\begin{equation*}
\lambda	( \omega_{k+1} -\omega_k + 2 \omega_k \omega_{k+1} - Q\omega_k \omega_{k+1}^2 ) \geq (-2\omega_{k+1} +  Q \omega_{k+1}^2)	(   2Q \sup_{i\geq k_0}  \triangle_i)+ 14CQ\omega_{k+1}^2 + 4Q\triangle_k  \omega_{k+1} 
	\end{equation*}
	Using  the fact that $\triangle_k-\sup_{i\geq k_0}  \triangle_i<0$, it is suffices to show that 
				\begin{equation*}
	\lambda	( C_3 - Q\omega_k) \omega_{k+1}^2 \geq   \omega_{k+1}^2(  C_4 + 2Q^2 \sup_{i\geq k_0}  \triangle_i)
	\end{equation*}
	where $\displaystyle{ C_3 = 	{\lim \inf}_{k\rightarrow \infty} 2\frac{\omega_k}{\omega_{k+1}} + \frac{\omega_{k+1}-\omega_k}{\omega_{k+1}^2} } $, $C_4 = 14CQ^2$.
	By choosing $\lambda_0$ and $k_0$ such that $\displaystyle{\omega_{k_0} \leq \frac{C_3}{2Q}}$, and $\displaystyle{\lambda_0 = \frac{4Q^2 \sup_{i\geq k_0}  \triangle_i + 2C_4}{C_3}}$, the desired inequality \eqref{eq:psi} holds.

\end{proof}

\begin{theorem} \label{thm:main_G}
	
	Suppose Assumptions 1-3 hold, the sequence $\{G_k, k=1,\cdots,\infty \}$ converge to $G_*$, and there exist a sufficiently large $k_0$ such that
	\begin{equation*}
	\mathbb{E}||G_k - G_*||^2 = \mathcal{O} (\lambda \omega_k + \sup_{i\geq k_0}	\mathbb{E} ||\delta(M,n,\epsilon_i)||) 
	\end{equation*}
	
\end{theorem}

	\begin{proof}
		Denote by $E_k = G_k-G_*$, we have
		\begin{equation}\label{eq:thm-error}
		||E_{k+1}||^2 = ||E_k||^2 + \omega_{k+1}^2|| \tilde{H}(\boldsymbol{\beta_{k+1}}, G_k) ||^2  + 2\omega_{k+1} \mathbb{E} \langle E_k, \tilde{H}(\boldsymbol{\beta_{k+1}}, G_k)  \rangle
		\end{equation}
		For the third term in \eqref{eq:thm-error}, we have
		\begin{align*} 
		\langle E_k, \tilde{H}(\boldsymbol{\beta_{k+1}}, G_k)  \rangle &\leq \langle E_k, H(\boldsymbol{\beta_{k+1}}, G_k)  + \delta(M,n,\epsilon_k)  \rangle\\
		& \leq \langle E_k, h(G_k) + \mu_{G_k}(\boldsymbol \beta_{k+1}) - \Pi_{G_k} \mu_{G_k}(\boldsymbol \beta_{k+1}) +\delta(M,n,\epsilon_k)  \rangle\\
		&\leq -||E_k||^2 +   \langle E_k, \mu_{G_k}(\boldsymbol \beta_{k+1}) - \Pi_{G_k} \mu_{G_k}(\boldsymbol \beta_{k}) \rangle  +  \langle E_k,\Pi_{G_k} \mu_{G_k}(\boldsymbol \beta_{k})- \Pi_{G_{k-1}} \mu_{G_{k-1}}(\boldsymbol \beta_k) \rangle  \\
		& \;\;\;  +  \langle E_k, \Pi_{G_{k-1}} \mu_{G_{k-1}}(\boldsymbol \beta_k)- \Pi_{G_k} \mu_{G_k}(\boldsymbol \beta_{k+1}) \rangle  +   ||E_k|| ||\delta(M,n,\epsilon_k)|| \\
		&:= -||E_k||^2 + (\text{\rom{1}}) + (\text{\rom{2}}) + (\text{\rom{3}})+   ||E_k|| \triangle_k ,
		\end{align*}
		where we use Lemma \ref{lemma:stability}, Assumption \ref{assumption:poisson}, and Cauchy-Schwarz in the second last step, and $||\delta(M,n,\epsilon_k)|| =  \triangle_k$.
		 
		 For (\rom{1}), we have $\mathbb{E} [ \mu_{G_k}(\boldsymbol \beta_{k+1}) - \Pi_{G_k} \mu_{G_k}(\boldsymbol \beta_{k}) | \mathcal{F}_k] = 0$, where $\mathcal{F}_k$ is a $\sigma$-filter formed by $\{ G_0, \boldsymbol \beta_1,G_1, \cdots, \boldsymbol \beta_k,G_k\}$.
		 
		 For (\rom{2}), by Assumption \ref{assumption:poisson}
		  \begin{equation}\label{eq:thm_rom2}
		\mathbb{E}  \langle E_k,\Pi_{G_k} \mu_{G_k}(\boldsymbol \beta_{k})- \Pi_{G_{k-1}} \mu_{G_{k-1}}(\boldsymbol \beta_k) \rangle  \leq C ||E_k|| || G_k-G_{k-1}|| \leq 4CQ^2 \omega_k  \leq 5CQ^2 \omega_{k+1}, 
		 \end{equation}
		where we use the fact that $|| G_k-G_{k-1}|| = || \omega_k  \tilde{H}(\boldsymbol{\beta_{k}}, G_{k-1})||  \leq  2  Q\omega_k $, and the last inequality in \eqref{eq:thm_rom2} use the assumption on the step size for a sufficient large number $k$.
		
		For (\rom{3}), by Assumption \ref{assumption:poisson}
		\begin{align*}
		&\langle E_k, \Pi_{G_{k-1}} \mu_{G_{k-1}}(\boldsymbol \beta_k)- \Pi_{G_k} \mu_{G_k}(\boldsymbol \beta_{k+1}) \rangle  = z_k - z_{k+1} +  \langle  E_{k+1} -E_k , \Pi_{G_k} \mu_{G_k}(\boldsymbol \beta_{k+1}) \rangle \\
		& \leq z_k - z_{k+1} + C|| E_{k+1} -E_k||   = z_k - z_{k+1} + C|| G_{k+1} -G_k||  \leq  z_k - z_{k+1} + 2CQ \omega_{k+1}
		\end{align*}
		where $z_k = \langle E_k, \Pi_{G_{k-1}} \mu_{G_{k-1}}(\boldsymbol \beta_{k}) \rangle $, $z_{k+1} = \langle E_{k+1}, \Pi_{G_k} \mu_{G_k}(\boldsymbol \beta_{k+1}) \rangle$.
		
		 Thus,
		 	\begin{equation*}
			\mathbb{E} ||E_{k+1}||^2 \leq (1-2\omega_{k+1} +Q\omega_{k+1}^2 )	\mathbb{E} ||E_k||^2 +  14CQ\omega_{k+1}^2 + 4Q \triangle_k \omega_{k+1}   + 2\omega_{k+1 }	\mathbb{E} [z_k - z_{k+1}]\rangle
		 \end{equation*}
		 
		 According to Lemma \ref{lemma:sequence}, there exists $\lambda_0$, $k_0$ such that
		 \begin{equation*}
		 \mathbb{E} ||E_{k_0}||^2 \leq \psi_{k_0} = \lambda_0 \omega_{k_0} + 2Q \sup_{i\geq k_0}  \triangle_i
		 \end{equation*}
		 Thus,
		 \begin{equation} \label{eq:thm_Ekbd}
		 \mathbb{E} ||E_{k}||^2 \leq \psi_{k} +  \mathbb{E}[\sum_{j=k_0+1}^k \Lambda_j^k (z_{j+1} - z_j)] 
		 \end{equation}
		 From Assumption \ref{assumption:poisson} and Lemma \ref{lemma:unifrom_G}, we have 
		  \begin{equation*}
		 \mathbb{E}[|z_k|]  = \mathbb{E}\left[\left|  \langle E_k, \Pi_{G_{k-1}} \mu_{G_{k-1}}(\boldsymbol \beta_{k}) \rangle  \right| \right] \leq  \mathbb{E} ||E_k||  \mathbb{E}\left[\left|  \Pi_{G_{k-1}} \mu_{G_{k-1}}(\boldsymbol \beta_{k})   \right| \right] \leq 2QC
		 \end{equation*}
		 
		 By Lemma \ref{lemma:Lamda}, 
		 \begin{align*}
		   \mathbb{E}\left[ \left| \sum_{j=k_0+1}^k \Lambda_j^k (z_{j+1} - z_j)  \right| \right] &=  \mathbb{E}\left[ \left| \sum_{j=k_0+1}^{k-1}  ( \Lambda_{j+1}^k - \Lambda_j^k) z_j +\Lambda_{k_0+1}^k  z_{k_0} -\Lambda_k^k  z_k \right| \right]  \\
		   &\leq  ( \Lambda_{k}^k - \Lambda_{k_0+1}^k) 2QC + 8QC \omega_k \leq 12QC \omega_k
		 \end{align*}
		 
Then the inequality \eqref{eq:thm_Ekbd} can be further bounded as
		 \begin{align*}
\mathbb{E} ||E_{k}||^2& \leq \lambda_0 \omega_k + 2Q \sup_{i \geq k_0}  \triangle_i + 12 QC  \omega_{k} \\
&= \lambda \omega_k + 2Q \sup_{i \geq k_0}  \triangle_i
\end{align*}
where $\lambda = \lambda_0 + 12 QC$.
	\end{proof}


\subsection{Weak convergence of model parameters}

Given a metric tensor $G(\boldsymbol {\beta }(t))$ on the manifold, the Langevin diffusion is characterized by
\begin{equation}\label{eq:langevin_manifold}
d \boldsymbol {\beta (t)} =    G(\boldsymbol {\beta } (t)) \left[  \nabla_{\boldsymbol \beta} {L} (\boldsymbol {\beta }  (t)) + \Gamma(\boldsymbol {\beta } (t))  \right] +  G^{\frac{1}{2}} (\boldsymbol {\beta } (t)) d \mathcal{B}_t
\end{equation}
where $\mathcal{B}_t$ is the standard Brownian motion.

Let $\mathcal{L}$ be the generator for \eqref{eq:langevin_manifold}, for any function $f$ which is compactly supported and twice differentiable,
\begin{equation}\label{eq:cont_generator}
\mathcal{L} f(\boldsymbol {\beta}(t) ) = \left( G(\boldsymbol {\beta}(t) ) \left[ \nabla_{\boldsymbol \beta} {L} (\boldsymbol {\beta }(t)  ) + \Gamma(\boldsymbol {\beta } (t))  \right]  \cdot \nabla_{\boldsymbol \beta}  +  \frac{1}{2} G^{\frac{1}{2}}(\boldsymbol {\beta}) G^{\frac{1}{2}}(\boldsymbol {\beta })^T : \nabla_{\boldsymbol \beta} \nabla_{\boldsymbol \beta}^T \right) f(\boldsymbol {\beta}(t) ) 
\end{equation}
where $\cdot$ denote the vector dot product, and $:$ denote the matrix double dot product, and generator $\mathcal{L}$ is associated with the backward Kolmogorov equation 
\begin{equation*} 
\mathbb{E} [f(\boldsymbol {\beta } (t) )] = e^{t\mathcal{L}} f(\boldsymbol {\beta}_0 )
\end{equation*}

In our work, we define the true generator using $G_*$ as 
\begin{equation}\label{eq:true_generator}
\mathcal{L}_* = G_*  \nabla_{\boldsymbol \beta} {L} (\boldsymbol {\beta }(t) )   \cdot \nabla_{\boldsymbol \beta}  +  \frac{1}{2} G_*  : \nabla_{\boldsymbol \beta} \nabla_{\boldsymbol \beta}^T  
\end{equation}
Given a test function $\phi$ of interest, let $\bar{\phi}$ be the posterior average of $\phi$ under the invariant measure of the associate SDE of \eqref{eq:true_generator}. Let $\boldsymbol {\beta}_{k}$ be numerical samples, and define 
$\hat{\phi} = \sum_{k=1}^{K} \frac{\epsilon_k}{S_K} \phi(\boldsymbol {\beta}_{k} )$, where $S_K = \sum_{k=1}^{K} \epsilon_k$. Let $\psi$ be a functional which solves the following Poisson equation 
\begin{equation*}
\mathcal{L}_* \psi(\boldsymbol {\beta}_{k}) = \phi(\boldsymbol {\beta}_{k}) - \bar{\phi}.
\end{equation*}
The solution functional characterize the difference between the posterior average and $\phi(\boldsymbol {\beta}_{k})$ for every $\boldsymbol {\beta}_{k}$. The assumption of $\psi$ is described as follows, which is the same as in \cite{sg-mcmc-convergence}.
\begin{assumption} \label{assump:psi}
	The functional $\psi$, and its derivatives $\mathcal{D}^j \psi$ ($j=1,2,3$), are bounded by a function $\mathcal{V}$. That is $||\mathcal{D}^j \psi|| \leq C_j \mathcal{V}^{p_j}$($j=0,1,2,3$), for some positive constants $C_j$ and $p_j$. Furthermore, $\mathcal{V}$ satisfies $\sup_k \mathbb{E}(\mathcal{V}(\boldsymbol {\beta}_{k})) < \infty$, and is smooth such that 
	\[
	 \sup_{s\in(0,1)} \mathcal{V}^p \left(s \boldsymbol {\beta} + (1-s) \boldsymbol {\gamma} \right)  \leq C \left(\mathcal{V}^p(\boldsymbol {\beta} ) +  \mathcal{V}^p (\boldsymbol {\gamma} ) \right)
	 \], 
	 $\forall \boldsymbol {\beta},  \boldsymbol {\gamma} $, and $p\leq \max \{2p_k\}$, $C>0$.
\end{assumption}

Next, we write the local integrator of our proposed method $\tilde{\mathcal{L}}_t$ as
\begin{equation}
\tilde{\mathcal{L}}_k = G(\boldsymbol {\beta}_k ) \left(\nabla_{\boldsymbol \beta} \tilde{L} (\boldsymbol {\beta}_k)  \right)  \cdot \nabla_{\boldsymbol \beta}  +  \frac{1}{2} G(\boldsymbol {\beta}_k) : \nabla_{\boldsymbol \beta} \nabla_{\boldsymbol \beta}^T
\end{equation}

Then $\tilde{\mathcal{L}}_k  =  \mathcal{L}_* + \Delta V_k$, with 
\begin{equation*}
\Delta V_k =   \left( G(\boldsymbol {\beta}_k) - G_* \right) \nabla_{\boldsymbol \beta} {L} (\boldsymbol {\beta}_k)  \cdot \nabla_{\boldsymbol \beta} +  \left( G(\boldsymbol {\beta}_k) - G_* \right) \xi_k \cdot \nabla_{\boldsymbol \beta}  + \frac{1}{2} \text{tr} \left[ (  G(\boldsymbol {\beta}_k) - G_*)^T \nabla_{\boldsymbol \beta}\nabla_{\boldsymbol \beta}^T \right]
\end{equation*}
where $\xi_k$ is the stochastic noise which comes from $\nabla_{\boldsymbol \beta} \tilde{L} (\boldsymbol {\beta}_k )-\nabla_{\boldsymbol \beta} {L} (\boldsymbol {\beta}_k)$.

We now state the estimates for the bias and MSE.
\begin{theorem}
	Under Assumptions \ref{assump:psi}, the bias and MSE of HAMCMC-SA for $K$ steps with decreasing step size $\epsilon_k$ is bounded,
	\begin{equation*}
	\text{Bias:   }   |\mathbb{E} \hat{\phi} - \bar{\phi} | =  O\Bigg( \frac{1}{S_K} 
	+  \sum_{k=1}^{K} \frac{\lambda \omega_k \epsilon_k}{S_K}  +  \sum_{k=1}^{K}  \frac{\epsilon_k^2 }{S_K}  + 2Q \sup_{i \geq k_0} \triangle_i \Bigg)
	\end{equation*}
	
\end{theorem}

\begin{proof}
Following a similar proof as in \cite{sg-mcmc-convergence}, one can obtain the following:
\begin{equation}\label{eq:diff_post_average}
\hat{\phi} - \bar{\phi}  =  \frac{1}{S_K} \left( \mathbb{E} \psi(\boldsymbol {\beta}_{L} ) - \psi(\boldsymbol {\beta}_{0} ) \right)
+  \frac{1}{S_K} \sum_{k=1}^{K-1}  \left( \mathbb{E} \psi(\boldsymbol {\beta}_{k} ) - \psi(\boldsymbol {\beta}_{k} )\right)
-  \sum_{k=1}^{K} \frac{\epsilon_k}{S_K} \Delta V_k  \psi(\boldsymbol {\beta}_{k-1} ) + C \frac{\epsilon_k^2 }{S_K}
\end{equation}
Taking expectation on both sides of \eqref{eq:diff_post_average}, 
\begin{equation}\label{eq:pre_bias}
\left| \mathbb{E} \hat{\phi} - \bar{\phi}  \right| \leq  \frac{1}{S_K} \mathbb{E} \left|  \psi(\boldsymbol {\beta}_{K} ) - \psi(\boldsymbol {\beta}_{0} ) \right|
+  \sum_{k=1}^{K} \frac{\epsilon_k}{S_K}  \left|  \mathbb{E}[ \Delta V_k  \psi(\boldsymbol {\beta}_{k-1} ) ]  \right|+ C \sum_{k=1}^{K} \frac{\epsilon_k^2 }{S_K}
\end{equation}

For the third term in the above equation,
\begin{align} \label{eq:deltaV_bound}
&\left|\mathbb{E} \left[ \Delta V_k  \psi(\boldsymbol {\beta}_{k-1}) \right] \ \right|  \\
 \leq  &\left| \mathbb{E}   \langle  \left( G(\boldsymbol {\beta}_k) - G_* \right) \nabla_{\boldsymbol \beta} {K} (\boldsymbol {\beta}_k),  \nabla_{\boldsymbol \beta}   \psi(\boldsymbol {\beta}_{k-1} ) \rangle \right| 
 + \frac{1}{2}  || G(\boldsymbol {\beta}_k) - G_* || ||\mathbb{E}   \Delta  \psi(\boldsymbol {\beta}_{k-1})|| 
\end{align}
where we use the fact that $\nabla_{\boldsymbol \beta} \tilde{L} (\boldsymbol {\beta}_k )$ is an unbiased estimator of $\nabla_{\boldsymbol \beta} {L} (\boldsymbol {\beta}_k)$, and $\text{tr}(AB) =||AB||_F \leq ||A||_F ||A||_F$ where $||\cdot||_F$ is the Frobenius norm and is abbreviate for $||\cdot||$.

According to Assumption \ref{assump:psi}, we have derivatives $\psi(\boldsymbol {\beta}_{k-1})$ are bounded, 
\[ \langle  \left( G(\boldsymbol {\beta}_k) - G_* \right) \nabla_{\boldsymbol \beta} {L} (\boldsymbol {\beta}_k),  \nabla_{\boldsymbol \beta}   \psi(\boldsymbol {\beta}_{k-1} ) \rangle  \leq C || G(\boldsymbol {\beta}_k) - G_* ||
\]
for some positive constant $C$, since $\nabla_{\boldsymbol \beta} {L} (\boldsymbol {\beta}_k)$ is also bounded.

By Theorem \ref{thm:main_G}, \eqref{eq:deltaV_bound} can be further bounded 
\begin{equation*}
\left|\mathbb{E} \left[ \Delta V_k  \psi(\boldsymbol {\beta}_{k-1}) \right] \ \right|  \leq C \mathbb{E} || G(\boldsymbol {\beta}_k) - G_* ||  \leq C(\lambda \omega_k + 2Q \sup_{i \geq k_0} \triangle_i)
\end{equation*}

Thus, 
\begin{align*}
\left| \mathbb{E} \hat{\phi} - \bar{\phi}  \right| &=  O\Bigg( \frac{1}{S_K} 
+  \sum_{k=1}^{K} \frac{\epsilon_k}{S_K} (\lambda \omega_k + 2Q \sup_{i \geq k_0} \triangle_i) +  \sum_{k=1}^{K}  \frac{\epsilon_k^2 }{S_K}  \Bigg)\\
&=  O\Bigg( \frac{1}{S_K} 
+  \sum_{k=1}^{K} \frac{\lambda \omega_k \epsilon_k}{S_K}  +  \sum_{k=1}^{K}  \frac{\epsilon_k^2 }{S_K}  + 2Q \sup_{i \geq k_0} \triangle_i \Bigg)
\end{align*}
where $\omega_k = \mathcal{O}(k^{-\alpha})$. As $K \rightarrow \infty$, $\left| \mathbb{E} \hat{\phi} - \bar{\phi}  \right| \rightarrow 2Q\sup_{i \geq k_0} \triangle_i $, which is a controllable bias.

As for the MSE, we following a similar proof as in \cite{sg-mcmc-convergence}, as long as $\sup_k \mathbb{E} || \Delta V_k  \psi(\boldsymbol {\beta}_{k-1}) ||^2$ is bounded, which is obvious, we have as $K \rightarrow \infty$, $\mathbb{E} \left( \hat{\phi} - \bar{\phi}  \right)^2  \rightarrow 0$.

\end{proof}


\section{Numerical examples}\label{sec:numerical}
In the last section, we will perform several numerical tests using proposed algorithm. 
\subsection{2D Gaussian distribution}
We first consider a simple 2D Gaussian distribution $\mathcal{N}(\mu, \Sigma)$ for a simple illustration, where $\displaystyle{\mu = (0,0)^T}$, $\displaystyle{\Sigma =  \begin{bmatrix}
\sigma_x^2& -0.95\sigma_x \sigma_y\\
-0.95\sigma_x \sigma_y\ & \sigma_y^2
\end{bmatrix}}$ and $\sigma_x = 0.12, \sigma_y = 1$. In such a case, the two random variables have different scales of uncertainty and are correlated. Given some posterior samples, we aim to estimate the covariance matrix. We compare the proposed method HASGLD-SA with vanilla SGLD. In Figure \ref{fig:2dgaussian} (a), we show the first 2500 samples generated from both methods, where we set burn-in to be 500. The contour of the true posterior is shown in the background. It shows that HASGLD-SA can explore the posterior better. In Figure \ref{fig:2dgaussian} (b), we compare two methods using different step sizes $\{0.02, 0.02\times0.8, 0.02\times0.8^2, 0.02\times0.8^4 \}$, and 30,000 samples are generated in each case. The  average absolute error of sample covariance vs autocorrelation time (ACT) are plotted. We can see that HASGLD-SA outperforms SGLD by showing a lower error and smaller autocorrelation time.

\begin{figure}[!h] 
	\captionsetup[subfigure]{justification=centering}
	\begin{subfigure}{.7\textwidth}
		\centering
		\includegraphics[scale=0.37]{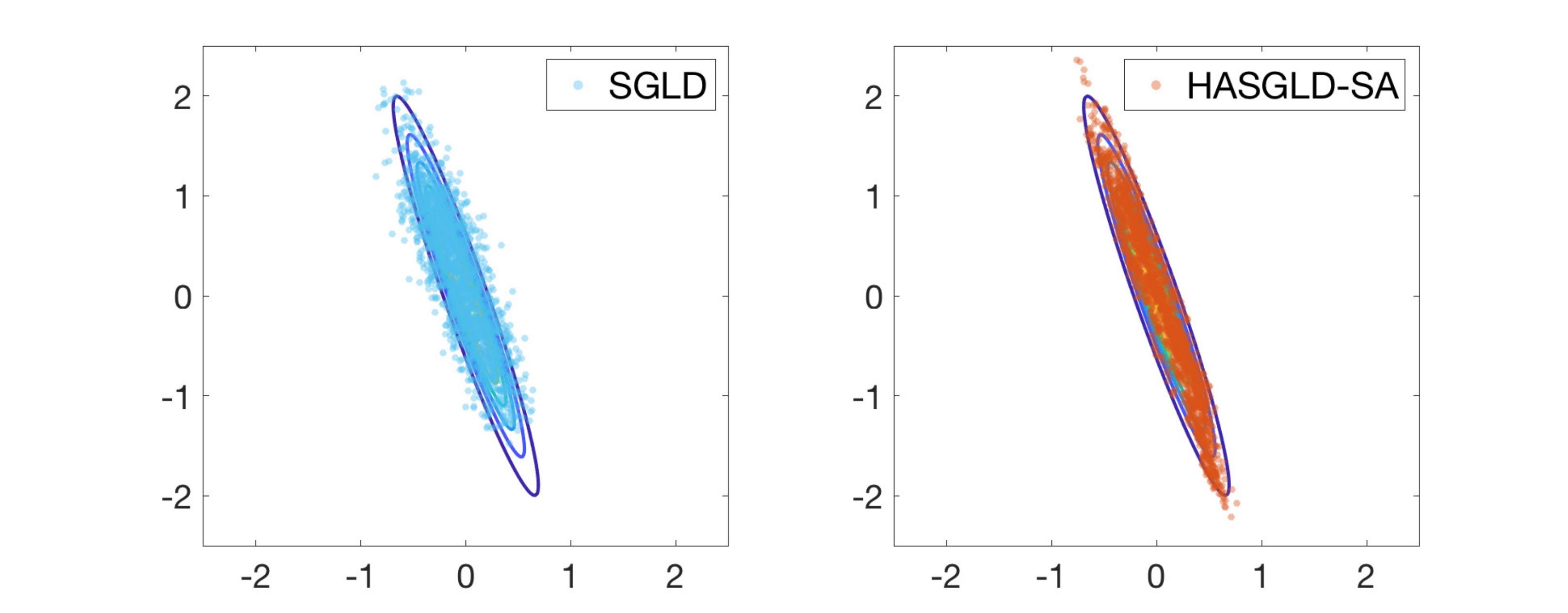}
		\caption{Samples obtained from SGLD and HASGLD-SA}
	\end{subfigure}\hfill
	\begin{subfigure}{.3\textwidth}
		\centering
		\includegraphics[scale=0.53]{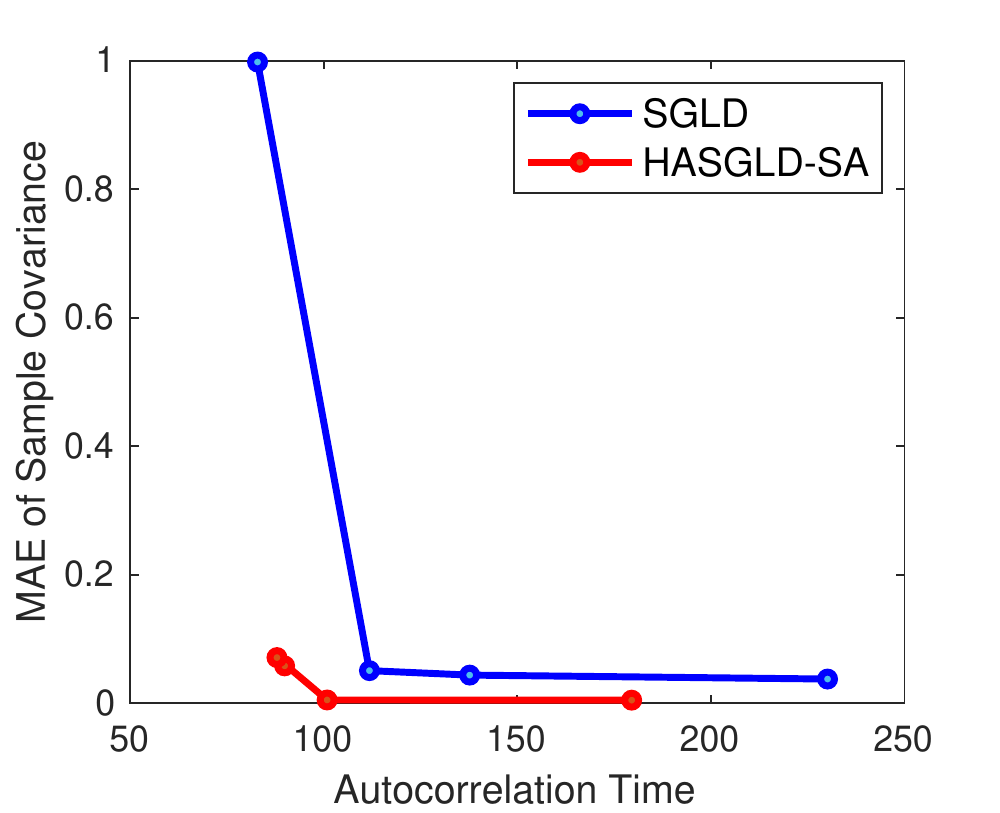}
		\caption{Covariance error and ACT}
	\end{subfigure}\hfill 
	\caption{2D Gaussian distribution. Subfigure (a) shows the comparison of samples obtained from SGLD and HASGLD-SA. Subfigure (b) shows the covariance error and autocorrelation time comparison between two methods.} \label{fig:2dgaussian}
\end{figure}

\subsection{Small $n$ large $p$ problem}

We then test on a linear regression problem with $n$ observations and $p$ model parameters, where $n<<p$. Let the model parameters be $\boldsymbol\beta \in \mathbb{R}^p$, $\beta_1 = 3, \beta_2 = 1, \beta_j = 0$, for $j=1, \cdots, p$. Denote by $X \in  \mathbb{R}^{n\times p}$ the predictors, which is generated from $\mathcal{N}_p(0, \Sigma)$ with $\Sigma_{ij} = 0.8^{\frac{1}{4}|i-j|}$.The responses $y = X \boldsymbol \beta + \epsilon$, and $\epsilon \sim \mathcal{N}_n(0, 3I_n)$.  In this example, we take $n = 100$ and $p = 200$.  We compare the performance of SGLD-SA and HASGLD-SA and present them in Figure \ref{fig:small-n-large-p}. We remark that, in this example, we assume the model parameter ${\boldsymbol\beta}_{j}$ follows a spike and slab Gaussian-Laplace prior in order to perform sparse inference. That is, ${\boldsymbol\beta}_{j}| \sigma^2, \gamma_j \sim  \gamma_j \mathcal{N} (0, \sigma^2 v_1) + (1-\gamma_j) \mathcal{L} (0, \sigma v_0)$, where $\gamma_j = \{0,1\}$. Similar as in \cite{sgld-sa}, the hyper-parameters priors are $\sigma \sim IG(\nu/2, \nu \lambda/2)$, $\pi(\gamma_j|\delta_j) = \delta_j^{|\gamma_j|}  (1-\delta_j)^{p_j-|\gamma_j|}$, and $\pi(\delta_j) = \delta_j^{a-1} (1-\delta_j)^{b-1}$. The priors will be learned through optimization. We choose $\nu=1, \lambda=1, v_1 = 100, v_0=0.1, \delta = 0.5, a=1, b=p$, and the step size for updating hyper-parameters in the priors to be $\omega_k =  5\times (10 + k)^{-0.9}$. The learning rate is chosen to be $0.1$. The comparison of posterior mean $\hat{\boldsymbol \beta}$ and true $\boldsymbol\beta$ is shown in the left subplot of Figure \ref{fig:small-n-large-p}. It shows that HASGLD-SA identifies the model parameters better. Moreover, for testing purposes, we generate 50 new samples, and use the estimated posterior mean in each step to perform a prediction. Then we compute the mean MSE and MAE error of the predicted responses with true responses among these testing samples, and show the results in Figure \ref{fig:small-n-large-p}. We observe that HASGLD-SA has consistently smaller errors during this process.

\begin{figure}[!hbt]
	\centering
	\includegraphics[scale=0.35]{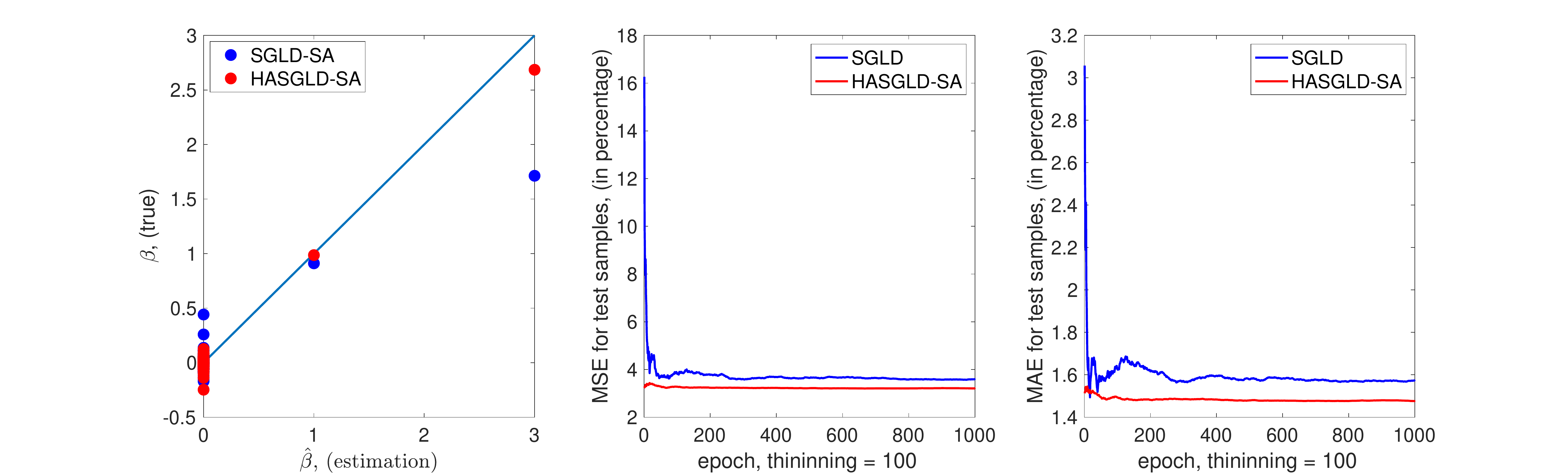}
	\caption{A comparison between three methods for large-p-small-n problem.}
	\label{fig:small-n-large-p}
\end{figure}

\subsection{Solutions of Elliptic PDE}\label{sec:num_ex2}
Next, we apply the proposed approaches to predict solutions the elliptic problem with heterogeneous permeability fields. The mixed formulation of the elliptic problem reads:
\begin{align*}
\kappa^{-1} u + \nabla p &= 0 \quad \quad \text{in}  \quad \Omega\\
\text{div} (u) &= f \quad \quad \text{in}  \quad \Omega
\end{align*}
where $\kappa$ represents permeability, $f$ is the source.
The domain$\Omega=[0,1]\times[0,1]$, and the boundary consists of $\partial \Omega = \Gamma_N  \cup \Gamma_D$. Raviart-Thomas element $\text{RT}_0$ and piecewise constant element $P_0$ pairs are chosen to solve the linear system, and the solution vectors will be used as training labels. The mixed finite element system on the fine grid has the matrix form
\begin{equation*}
\begin{bmatrix}
A_h & B_h^T   \\
B_h & 0
\end{bmatrix}
\begin{bmatrix}
u_h \\
p_h
\end{bmatrix} =
\begin{bmatrix}
D_b  \\
-F
\end{bmatrix}
\end{equation*}
where where $[A_h]_{ij} = \int_{\Omega} \kappa^{-1} \psi_i \cdot \psi_j$, and $[B_h]_{ij} = -\int{\Omega} p_k \; \text{div} \psi_j $, where $\psi_j$ is the velocity basis on the $i$-th fine scale edge, $p_k$ is the pressure basis on the $k$-th fine scale block.

It is well known that the multiscale properties of the permeability fields require very fine-scale meshes to recover all scale information. Numerous methods have been proposed to develop reduced-order models to alleviate the computational burden. A popular class of approaches among these includes the mixed multiscale finite element method \cite{ch02, MixedGMsFEM}. The idea is to construct a multiscale velocity basis by solving some local problems on each coarse region and couple them with a mixed formulation. If the underlying permeability has rich information, several multiscale bases are needed to capture these features to provide an accurate approximation. The mixed FEM formulation on the coarse grid level preserves mass conservative property which is essential for flow problems. 

To be specific, denote by $N_u^H$ be dimension of the multiscale velocity solution space, and let $R_u \in \mathbb{R}^{N_u^H \times N_u^h} $ be the matrix with these velocity basis in every row, where $N_u^h$ is the dimension of fine scale velocity solution space. Similarly, denote by $R_p$ the matrix containing piecewise constant basis on coarse grid level which maps fine scale pressure vector in $\mathbb{R}^{N_p^h}$ to coarse scale pressure vector in $\mathbb{R}^{N_p^H}$. The mixed formulation on the coarse grid reads
\begin{equation*}
\begin{bmatrix}
A_H & B_H^T   \\
B_H & 0
\end{bmatrix}
\begin{bmatrix}
u_H \\
p_H
\end{bmatrix} =
\begin{bmatrix}
R_u & 0   \\
0 & R_p
\end{bmatrix}
\begin{bmatrix}
A_h(\kappa) & B_h^T   \\
B_h & 0
\end{bmatrix}
\begin{bmatrix}
R_u^T & 0   \\
0 & R_p^T
\end{bmatrix}
\begin{bmatrix}
u_H \\
p_H
\end{bmatrix}=
\begin{bmatrix}
0   \\
-F_H
\end{bmatrix}
\end{equation*}
One can observe that $\begin{bmatrix}
R_u & 0   \\
0 & R_p
\end{bmatrix}$ performs an upscaling procedure which is analogy to an encoder, and $\begin{bmatrix}
R_u^T & 0   \\
0 & R_p^T
\end{bmatrix}$ acts as downscaling matrix which can be viewed as a decoder.

After one obtains the coarse-scale solution vector $u_H$ from the above system, the multiscale solution $u_{\text{ms}}$ can be recovered using $u_{\text{ms}} = \sum_{i=1}^{N_u^H} (u_H)_i \Psi_i$, where $(u_H)_i$ is the $i$-th component in $u_H$, and $\Psi_i$ is the $i$-th column in $R_u^T$.
To obtain an accurate approximation $u_{\text{ms}}$ to $u_h$, it is crucial to design good local problems and basis selecting algorithms which are used for solving multiscale bases.  Moreover, many practical applications need to solve the flow problem with (1) varying source terms or boundary conditions, given a fixed permeability field, or (2) different permeability fields. In the second case, the multiscale basis needs to be reconstructed every time providing a new $\kappa$. To avoid these technical difficulties, we aim to borrow the upscaling-downscaling idea from coarse grid solvers, and construct an encoding-decoding type of neural network \cite{wang_multiphase} as surrogate models (1) between the source term $f$ and fine grid velocity solution $u_h$, (2) between the permeability fields $\kappa$ and fine grid velocity solution $u_h$. We refer to \cite{wang_multiphase} for the details of the network architecture.

\subsubsection{Varying source term}
we first consider the case when $f$ are different among samples, but the $\kappa$ is a fixed permeability field from SPE10 model. We use a three-spot source term, where the three blocks with nonzero source lie in the center $\omega_c \in \Omega$, the upper right corner $\omega_{up} \in \Omega$ and lower left corner $\omega_{ll} \in \Omega$ of the computational domain. The values of the source is set to be  
\begin{equation*}
f(x) = \begin{cases}
f_1 \sim \displaystyle{\mathcal{N}(10, 5) }, & \text{if } \displaystyle{ x \in \omega_{up} } \\
f_2 \sim \displaystyle{\mathcal{N}(10, 5) }, & \text{if } \displaystyle{ x \in \omega_{ll} } \\
-(f_1+f_2), & \text{if } \displaystyle{ x \in \omega_{c} } \\
0 & \text{otherwise}
\end{cases} 
\end{equation*} 

 An illustration of the permeability field, source term and corresponding velocity solution is shown in Figure \ref{fig:ex_3-1_illustration}.
\begin{figure}[!h]
	\includegraphics[width=.25\textwidth]{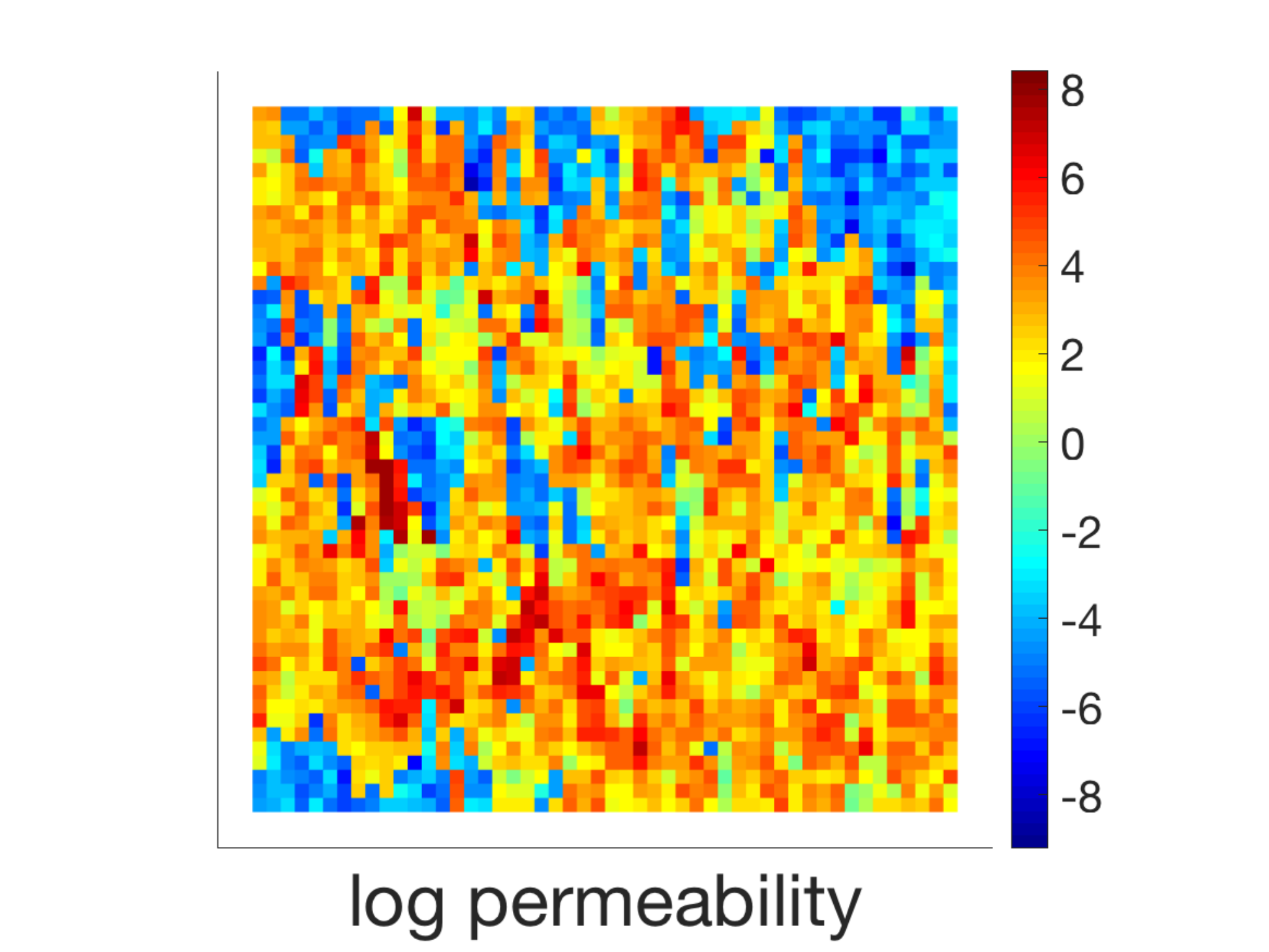}\hfill
	\includegraphics[width=.25\textwidth]{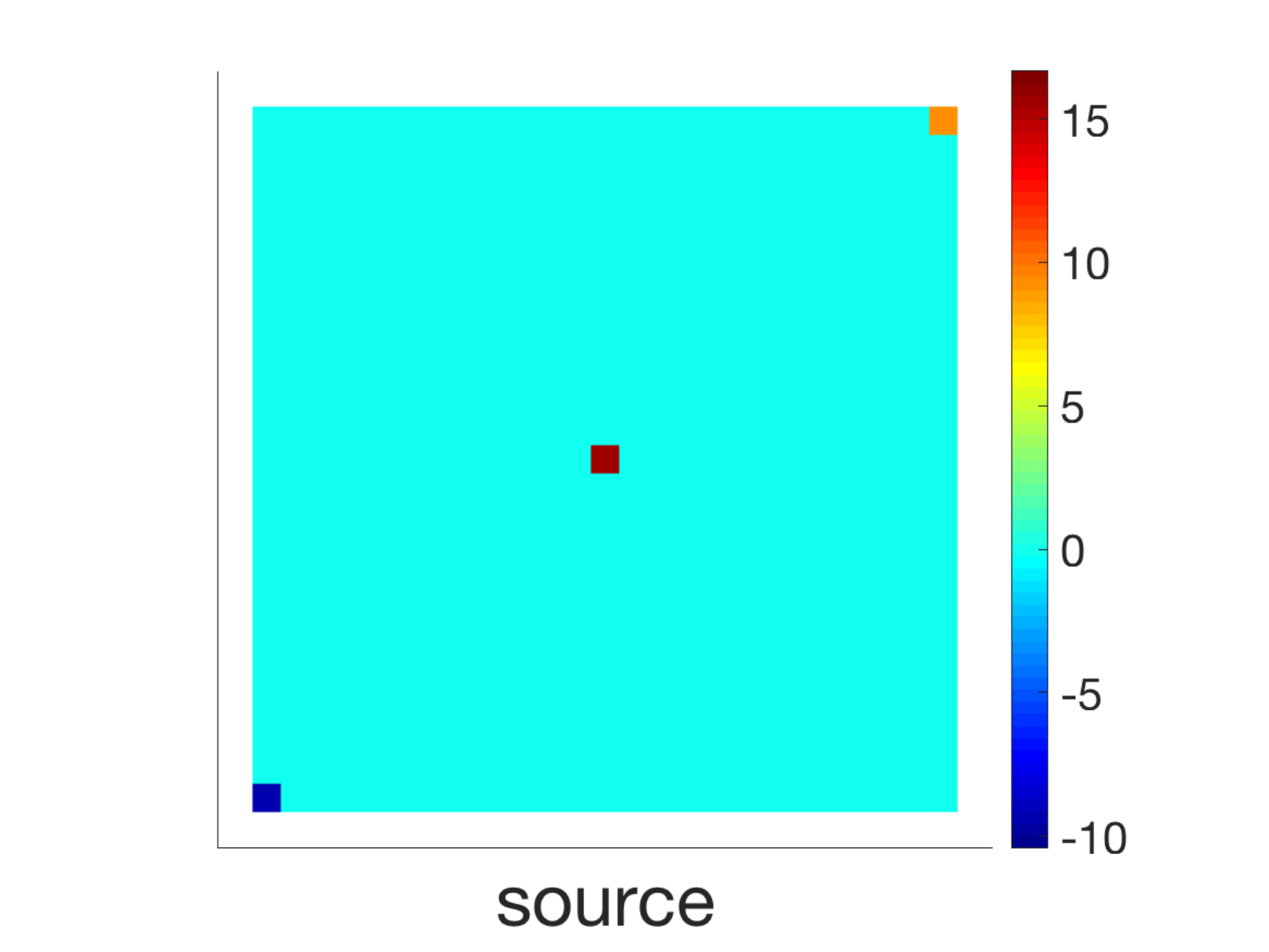}\hfill
	\includegraphics[width=.25\textwidth]{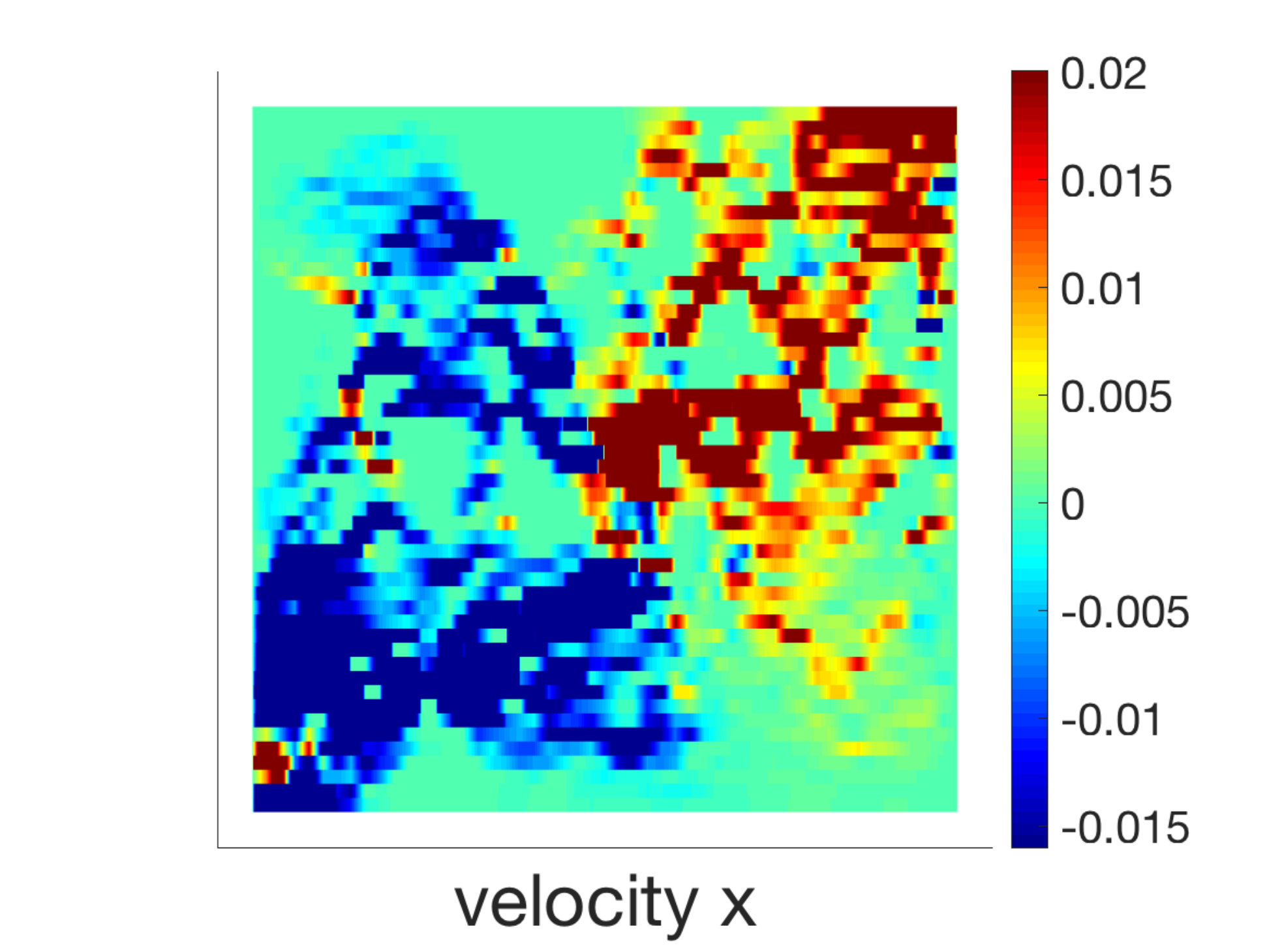}\hfill
    \includegraphics[width=.25\textwidth]{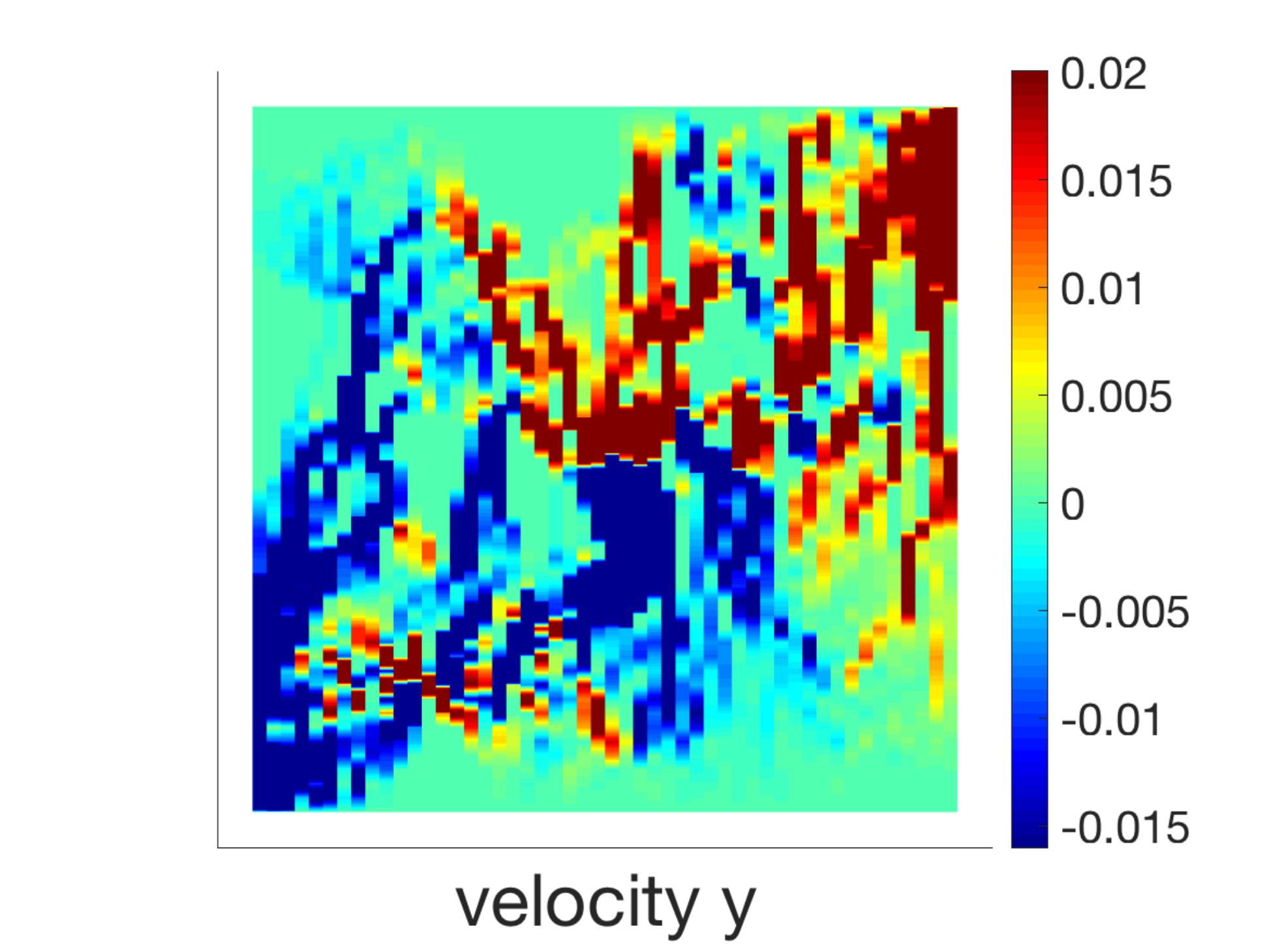}\hfill
	\caption{ From left to right:  The permeability field of SPE10 model (in log scale);  A three-spot source; Velocity solution magnitude in $x$ direction; Velocity solution magnitude in $y$ direction.}
	\label{fig:ex_3-1_illustration}
\end{figure}

We run the simulation for $1500$ different source terms and use the source-velocity pairs to train the neural network $\mathcal{F}$, where $\mathcal{F}(f) \approx u$. $80\%$ of the samples are randomly selected to train the network and the rest $20\%$ will be used for testing. The architecture of the network is as follows. The first layer is an average pooling layer with pool size $2\times 2$, a flatten layer is followed to transform the image into its vector version, then a fully connected layer with $100$ neurons is adopted. This part of the network encodes the input and is in analogy to upscaling. Then we reshape this intermediate output to square images, use another two convolution layers, a flatten layer, and a fully connected layer with $200$ neurons to extract more hidden features. Finally, a dense layer is used to decode the features. The network has $2,566,828$ weight parameters in total.

We use the relative $l_2$ error in the loss function
 \begin{equation*}
{ || u_{i} - \mathcal{F}(f_i) || =  \frac{|| u_{i} - \mathcal{F}(f_i) ||_2 }{ ||u_{i} ||_2} }
\end{equation*}
where $u_i$ is the true velocity solution obtained from mixed FEM solver, $\mathcal{F}(f_i)$ is the neural network prediction for the $i$-th sample. The mean errors for testing are shown in Table \ref{tab:ex_3-1}. We see that with $1$ or $2$ memory size, HASGLD-SA gives smaller errors consistently compared with vanilla SGLD. A few sample comparisons are shown in Figure \ref{fig:ex_3-1}. We remark that these are some bad predictions in the testing set, for other sample predictions, the errors are small and the discrepancies cannot be visualized obviously. We observe that, SGLD predictions lose some features compared with true solution, while  HASGLD-SA captures the heterogeneities in the solution well.

\begin{table}[!htb]
	\centering
	\begin{tabular}{|c  |c  | c  | c | c |}
		\hline
			
		&SGLD &HASGLD-SA (M=1) &HASGLD-SA (M=2)\\ \hline	
		No pruning &2.03 & 0.45 &0.42    \\  \hline
		Pruning Sparse rate 30\%  &1.38 &0.37 &0.34  \\  \hline
		Pruning Sparse rate 50\% &1.25 &0.29 &0.27  \\  \hline
		Pruning Sparse rate 70\%  &1.26 &0.30&0.27 \\  \hline
	\end{tabular}
	\caption{Mean errors (in percentage) for $300$ testing samples among the true and predicted solutions using proposed HASGLD-SA with memory size $M=1$, $M=2$, and SGLD.}\label{tab:ex_3-1}
\end{table}

	\begin{figure}[!hbt]

	\begin{subfigure}[t]{1.0\textwidth}
	\centering
	\includegraphics[width=0.8\textwidth]{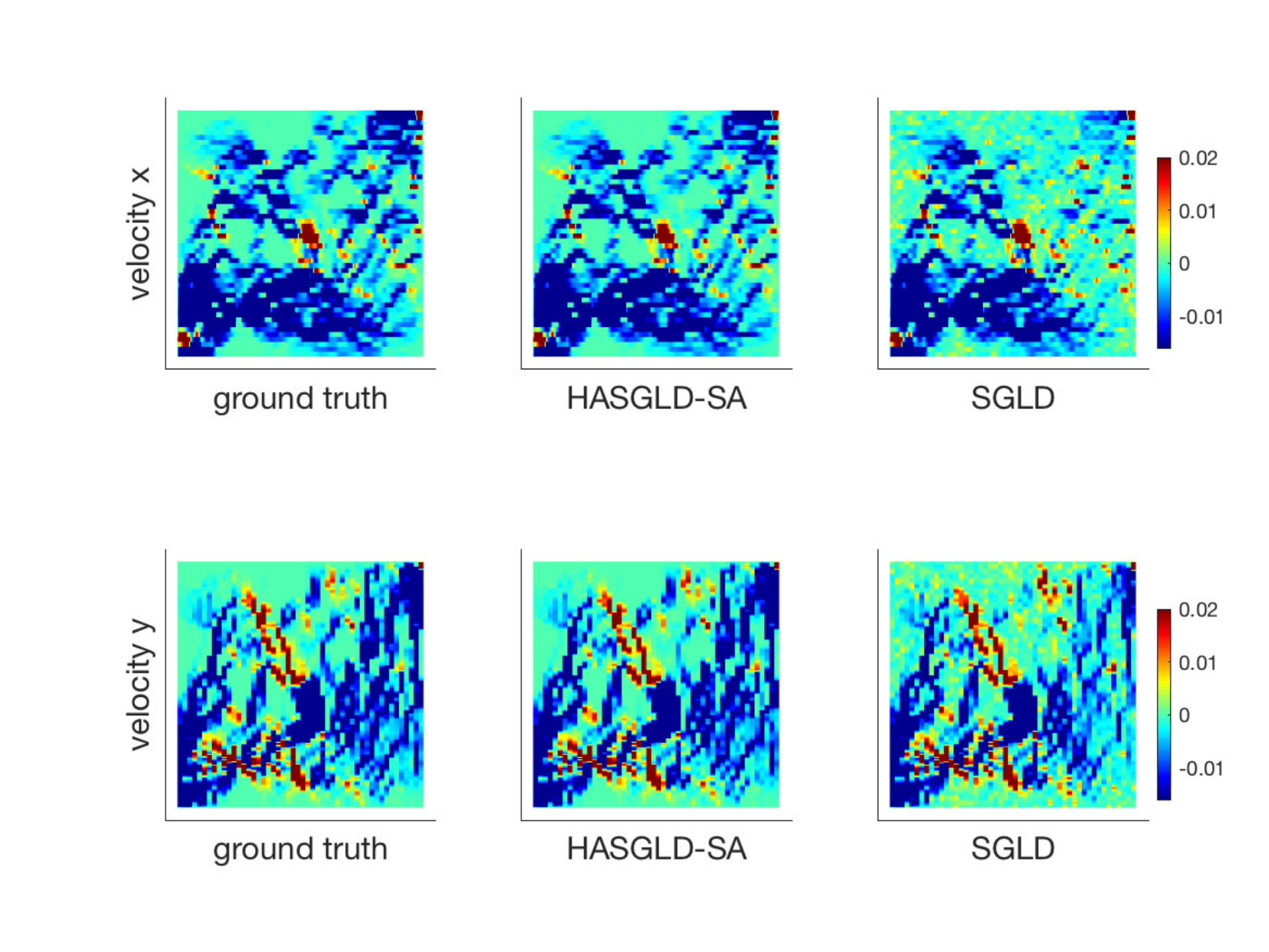}
		\caption{Test case 1}
	\end{subfigure}
	\begin{subfigure}[t]{1.0\textwidth}
	\centering
	\includegraphics[width=0.8\textwidth]{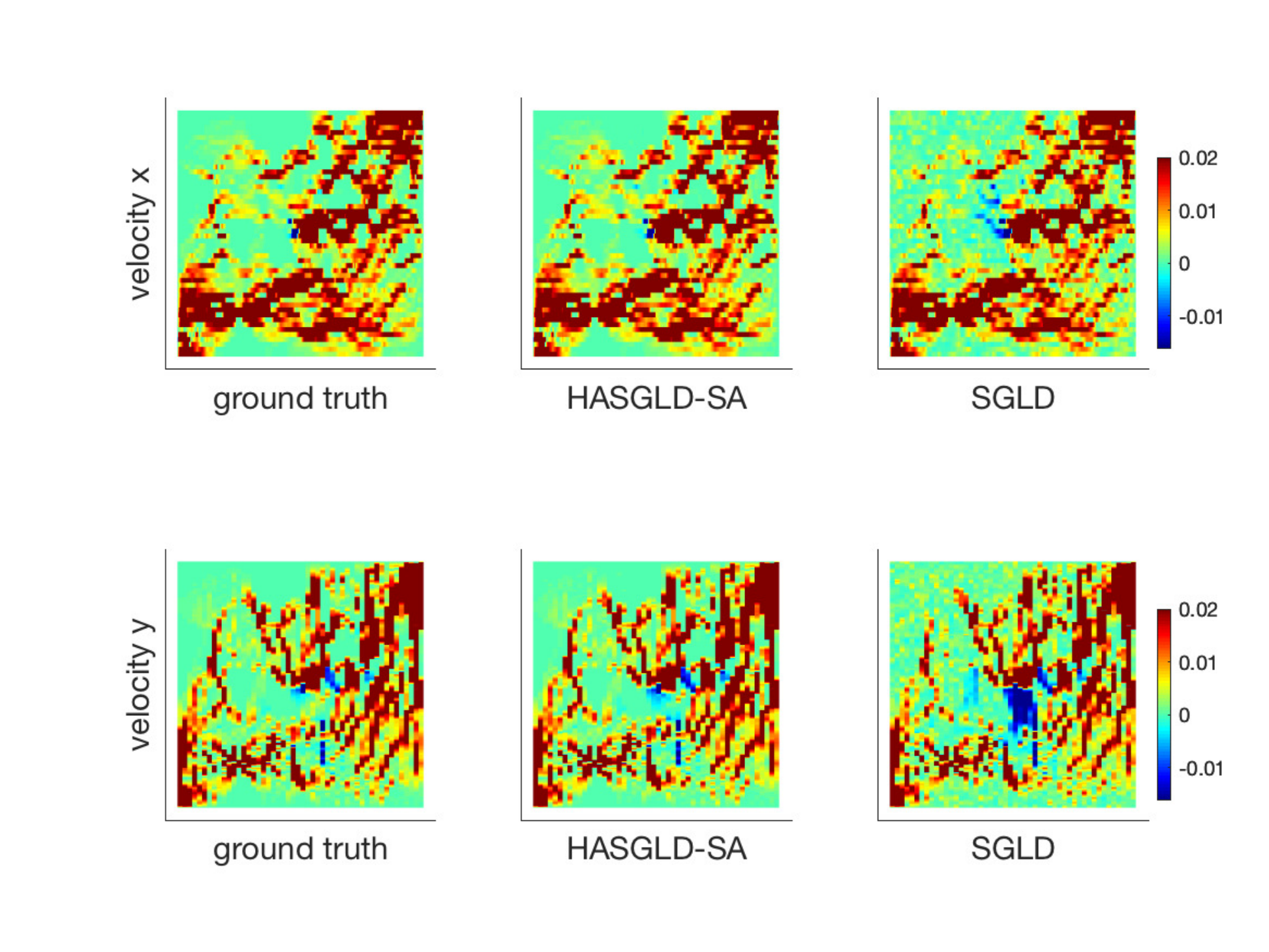}
		\caption{Test case 2}
	\end{subfigure}
\caption{Comparison of true and predicted solutions}\label{fig:ex_3-1}
\end{figure}

\subsubsection{Varying heterogeneous coefficients}

 In this section, we consider the case when heterogeneous coefficients vary and let $f = 1$ be a constant source term. The boundary conditions are $u\cdot n  = 0$ on the top and bottom sides of the square domain, $p = 1$ on the left boundary, and $p = 0$ on the right boundary.

 $\kappa$ can be obtained using Karhunen-Loeve expansion as follows:
\[
\kappa(x; \mu) = \kappa_0 +\displaystyle{\sum_{j=1}^{p}} \mu_j  \sqrt{\xi_j} \Phi_j(x)
\] 
where $ \kappa_0$ is a constant which is the mean of the random field. Moreover, random variables $\mu_j$ are drawn from i.i.d $N(0,1)$. $(\sqrt{\xi_j}, \Phi_j(x))$ are the eigen-pairs obtained from a Gaussian covariance kernel:
\begin{equation*}
\text{Cov} (x_i, y_i; x_j, y_j) = \sigma \exp(\frac{|x_i -x_j|^2}{l_x^2} -\frac{|y_i -y_j|^2}{l_y^2}  )
\end{equation*}
where we choose $[l_x, l_y]=[0.2, 0.3]$, $\sigma = 2$ and $p = 64$ in our example.


The training and testing data for deep learning can be generated by solving the equations with MFEM for various permeability fields. An illustrations of the permeability fields for $p=32, 64, 128$ and corresponding their corresponding solutions are presented in \ref{fig:kle_sol}. We can see that when $p$ becomes larger, the velocity solutions exhibit many more scale features. 

\begin{figure}[!hbt]
			\begin{subfigure}[t]{1.0\textwidth}
		\centering
		\includegraphics[width=.30\textwidth]{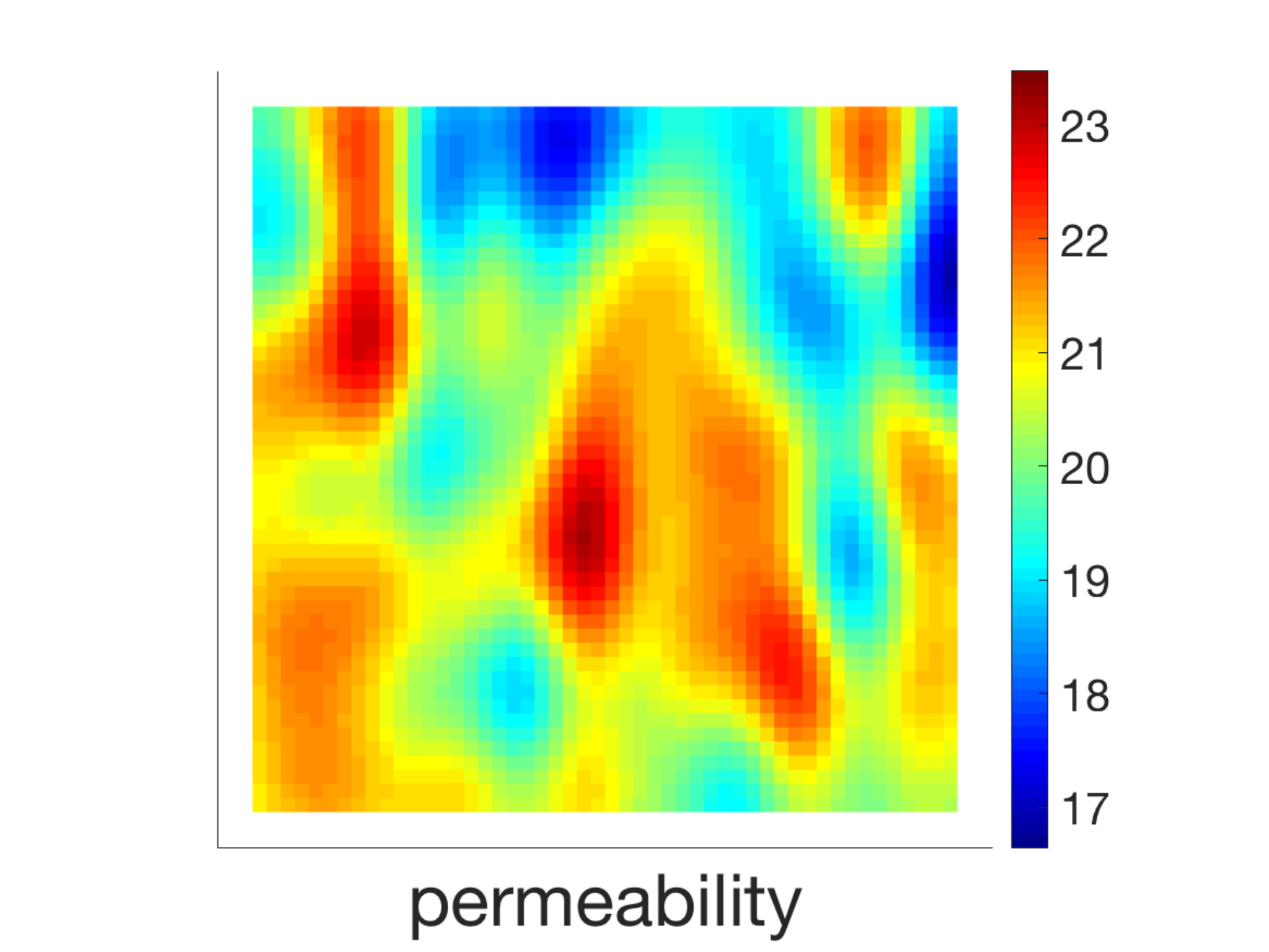}\;\;
		\includegraphics[width=.30\textwidth]{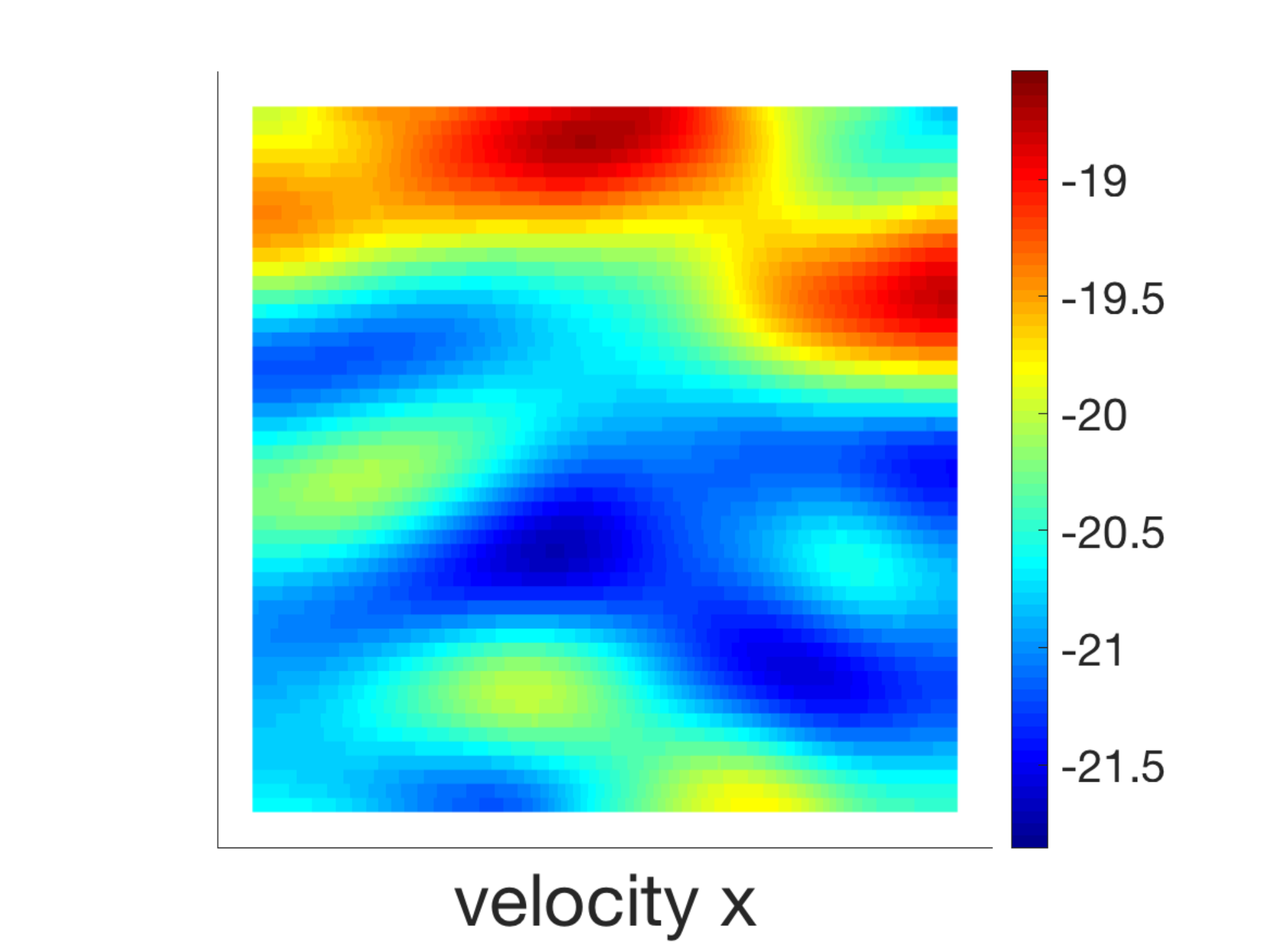}\;\;	
		\includegraphics[width=.30\textwidth]{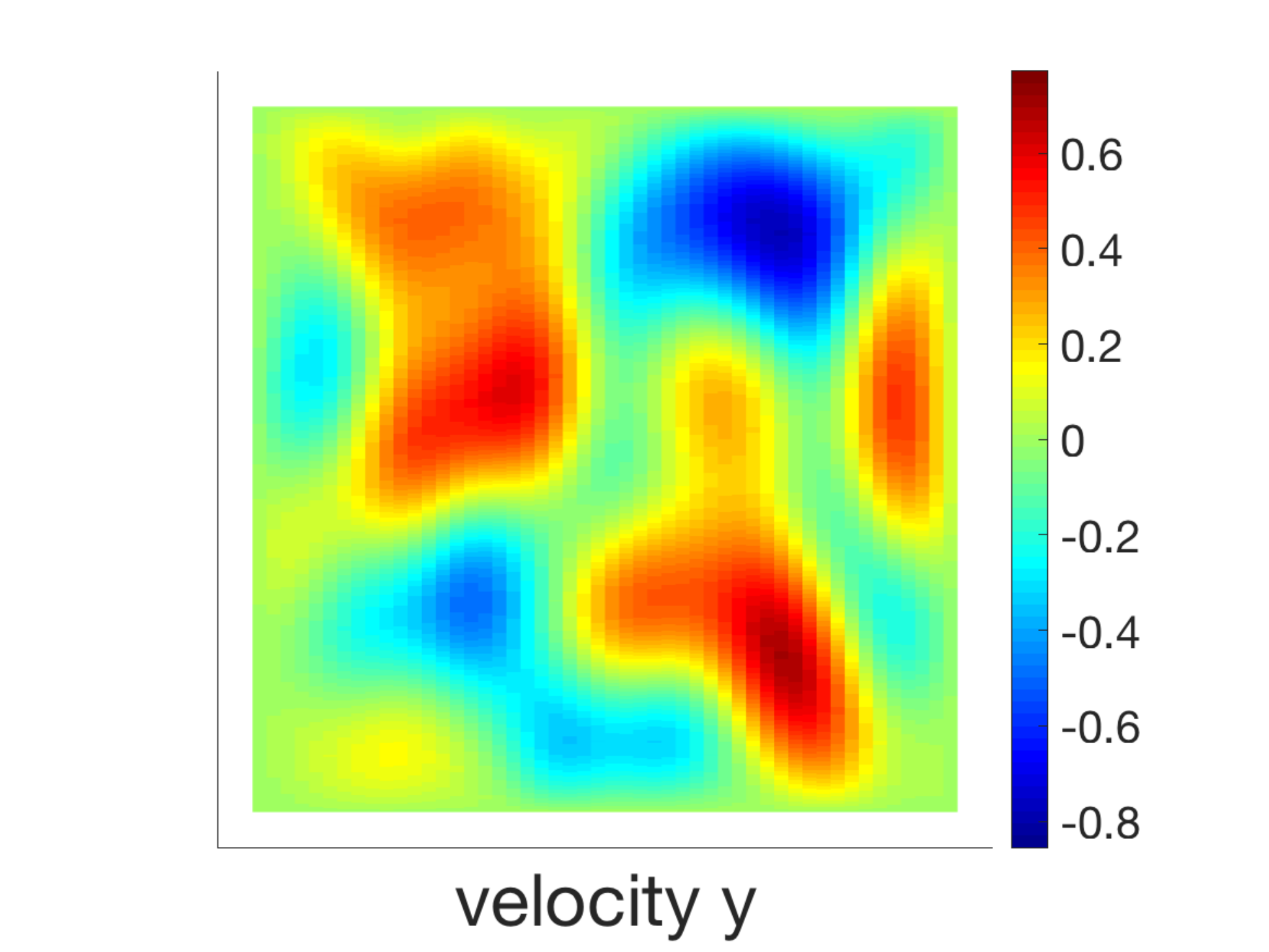}
	\end{subfigure}
	\caption{Illustrations of the permeability fields when using $64$ terms in KLE expansion and corresponding solutions. From left to right: Permeability, horizontal velocity magnitude, and vertical velocity magnitude.}\label{fig:kle_sol}
\end{figure}

We generate $1,500$ samples pairs $(\kappa_i, u_h^i)$, and randomly pick $1,300$ of them for training, and take the rest for testing. The size of an input permeability is $50\times 50$, an output velocity solution vector is $5,100$. The network consists of 2 convolution layers with kernel size $3\times 3$, and $64$ and $32$ channels, respectively. Then, an average pooling layer with pool size $2\times 2$ is followed by a flatten layer and then a dense layer with $100$ neurons. This part of the network can be viewed as an encoder. Then, a reshaping layer, another two convolution layers, a flatten layer, and a fully connected layer with $800$ neurons are used to mimic the coarse grid solver. Finally, a fully connected layer is used as a decoder. The total number of parameters is $8$,$252$, and $320$.

The numerical results using SGLD and HASGLD-SA are presented in Table \ref{tab:ex_3-2}. As an illustration, predictions of two samples are presented in Figure \ref{fig:ex_3-2}. The predictions obtained from vanilla SGLD are not reliable, and HASGLD-SA produces much better results.


\begin{table}[!htb]
	\centering
	\begin{tabular}{|c  |c  | c  | c | c |}
		\hline
	
		&\multirow{2}{*}{SGLD}  &HASGLD-SA &HASGLD-SA \\ 
		  & & (M=1) & (M=2) \\  \hline 
		No pruning &3.07 &2.72&1.68  \\  \hline
		Pruning Sparse rate 30\%  &3.04 &0.85 &0.78   \\  \hline  
		Pruning Sparse rate 50\% &3.06 &1.42 & 1.21 \\  \hline 
	\end{tabular}
	\caption{Mean errors among $300$ testing samples between the true and predicted solutions using proposed HASGLD-SA and SGLD.  }\label{tab:ex_3-2}
\end{table}

	\begin{figure}[!hbt]
	
	\begin{subfigure}[t]{1.0\textwidth}
		\centering
		\includegraphics[width=0.9\textwidth]{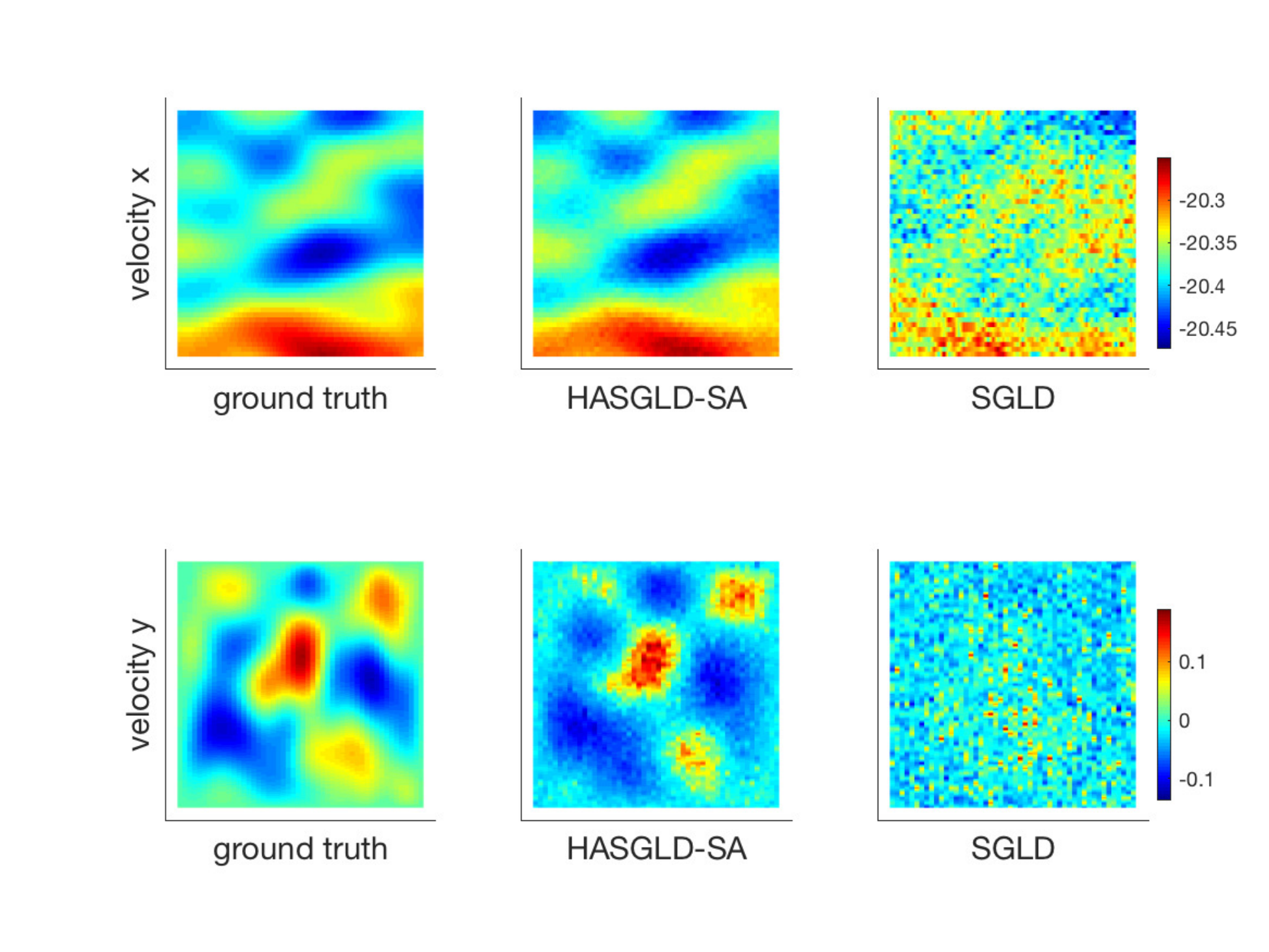}
		\caption{Test case 1}
	\end{subfigure}
	\begin{subfigure}[t]{1.0\textwidth}
		\centering
		\includegraphics[width=0.9\textwidth]{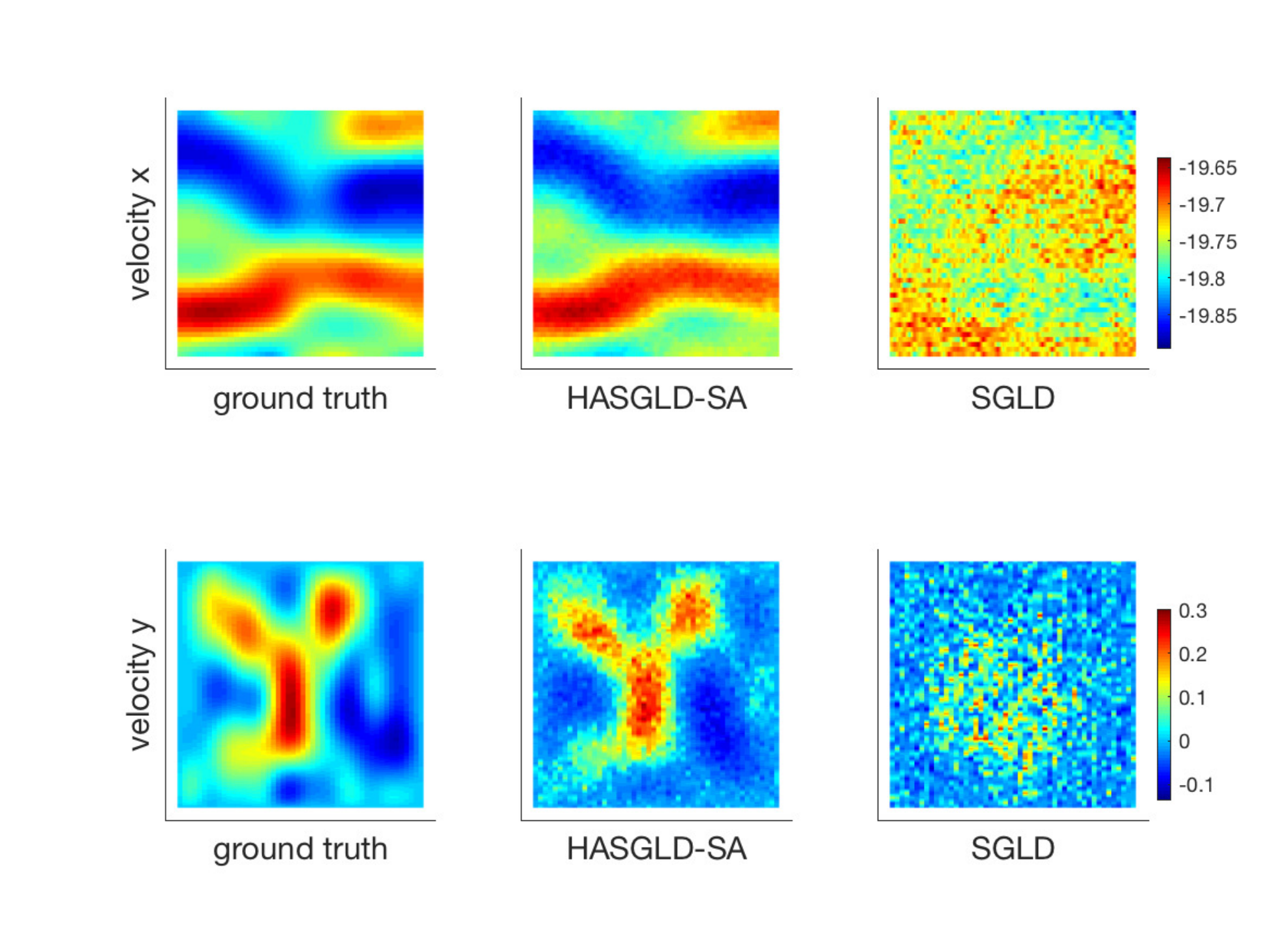}
		\caption{Test case 2}
	\end{subfigure}

	\caption{Comparison of true and predicted solutions}\label{fig:ex_3-2}
\end{figure}

\section{Conclusion} \label{sec:conclusion}
In this work, we proposed an adaptive Hessian approximated stochastic gradient MCMC method where the parameters are sampled from a posterior lying on a Riemannian manifold. The preconditioning matrix contains geometric information of the underlying density function and is updated via stochastic approximation in each iteration. It includes an approximation to the inverse Hessian which can be efficiently computed using a limited memory BFGS algorithm. We provide an analysis of the convergence of the proposed method and show that there is a controllable bias introduced by stochastic approximation. The bias term is generated due to the use of mini-batch when estimating the gradients, and the memory size which is used to approximate the inverse Hessian. It is expected to decrease if the batch size and the memory size are increased and if the step size in stochastic approximation and learning rate is decreased. In practice, our proposed algorithm achieves faster convergence and provides accurate predictions. In the future, we will explore applications of our proposed method to sparse deep learning.

\section*{Acknowledgement}
We gratefully acknowledge the support from the National Science Foundation (DMS-1555072, DMS-1736364, CMMI-1634832, and CMMI-1560834), Brookhaven National Laboratory Subcontract 382247, ARO/MURI grant W911NF-15-1-0562 and Department of Energy DE-SC0021142.
\bibliographystyle{siam} 
\bibliography{references.bib}

\begin{thebibliography}{10}

\bibitem{ahn2012bayesian}
{\sc S.~Ahn, A.~Korattikara, and M.~Welling}, {\em Bayesian posterior sampling
  via stochastic gradient fisher scoring}, arXiv preprint arXiv:1206.6380,
  (2012).

\bibitem{bordes2009sgd}
{\sc A.~Bordes, L.~Bottou, and P.~Gallinari}, {\em Sgd-qn: Careful quasi-newton
  stochastic gradient descent}, Journal of Machine Learning Research, 10
  (2009), pp.~1737--1754.

\bibitem{byrd2016stochastic}
{\sc R.~H. Byrd, S.~L. Hansen, J.~Nocedal, and Y.~Singer}, {\em A stochastic
  quasi-newton method for large-scale optimization}, SIAM Journal on
  Optimization, 26 (2016), pp.~1008--1031.

\bibitem{sg-mcmc-convergence}
{\sc C.~Chen, N.~Ding, and L.~Carin}, {\em On the convergence of stochastic
  gradient mcmc algorithms with high-order integrators.}, In Advances in Neural
  Information Processing Systems,  (2015), pp.~2278--2286.

\bibitem{chen2014stochastic}
{\sc T.~Chen, E.~Fox, and C.~Guestrin}, {\em Stochastic gradient hamiltonian
  monte carlo}, in International conference on machine learning, 2014,
  pp.~1683--1691.

\bibitem{ch02}
{\sc Z.~Chen and T.~Hou}, {\em A mixed multiscale finite element method for
  elliptic problems with oscillating coefficients}, Mathematics of Computation,
  72 (2002), pp.~541--576.

\bibitem{MixedGMsFEM}
{\sc E.~Chung, Y.~Efendiev, and C.~Lee}, {\em Mixed generalized multiscale
  finite element methods and applications}, SIAM Multicale Model. Simul., 13
  (2014), pp.~338--366.

\bibitem{dauphin2015equilibrated}
{\sc Y.~Dauphin, H.~De~Vries, and Y.~Bengio}, {\em Equilibrated adaptive
  learning rates for non-convex optimization}, in Advances in neural
  information processing systems, 2015, pp.~1504--1512.

\bibitem{dauphin2014identifying}
{\sc Y.~N. Dauphin, R.~Pascanu, C.~Gulcehre, K.~Cho, S.~Ganguli, and
  Y.~Bengio}, {\em Identifying and attacking the saddle point problem in
  high-dimensional non-convex optimization}, in Advances in neural information
  processing systems, 2014, pp.~2933--2941.

\bibitem{sgld-sa}
{\sc W.~Deng, X.~Zhang, F.~Liang, and G.~Lin}, {\em An adaptive empirical
  bayesian method for sparse deep learning.}, In Advances in Neural Information
  Processing Systems,  (2019), pp.~5564--5574.

\bibitem{ding2014bayesian}
{\sc N.~Ding, Y.~Fang, R.~Babbush, C.~Chen, R.~D. Skeel, and H.~Neven}, {\em
  Bayesian sampling using stochastic gradient thermostats}, in Advances in
  neural information processing systems, 2014, pp.~3203--3211.

\bibitem{Riemann_LD_HMC}
{\sc M.~Girolami and B.~Calderhead}, {\em Riemann manifold langevin and
  hamiltonian monte carlo methods.}, Journal of the Royal Statistical Society:
  Series B (Statistical Methodology), 73 (2011), pp.~123--214.

\bibitem{PSGLD}
{\sc C.~Li, C.~Chen, D.~Carlson, and L.~Carin}, {\em Preconditioned stochastic
  gradient langevin dynamics for deep neural networks.}, In Thirtieth AAAI
  Conference on Artificial Intelligence,  (2016).

\bibitem{liu1989limited}
{\sc D.~C. Liu and J.~Nocedal}, {\em On the limited memory bfgs method for
  large scale optimization}, Mathematical programming, 45 (1989), pp.~503--528.

\bibitem{ma2015complete}
{\sc Y.-A. Ma, T.~Chen, and E.~Fox}, {\em A complete recipe for stochastic
  gradient mcmc}, in Advances in Neural Information Processing Systems, 2015,
  pp.~2917--2925.

\bibitem{mokhtari2015global}
{\sc A.~Mokhtari and A.~Ribeiro}, {\em Global convergence of online limited
  memory bfgs}, The Journal of Machine Learning Research, 16 (2015),
  pp.~3151--3181.

\bibitem{SGRLD}
{\sc S.~Patterson and Y.~W. Teh.}, {\em Stochastic gradient riemannian langevin
  dynamics on the probability simplex.}, In Advances in neural information
  processing systems,  (2013), pp.~3102--3110.

\bibitem{robbins1951stochastic}
{\sc H.~Robbins and S.~Monro}, {\em A stochastic approximation method}, The
  annals of mathematical statistics,  (1951), pp.~400--407.

\bibitem{simsekli2016stochastic}
{\sc U.~Simsekli, R.~Badeau, T.~Cemgil, and G.~Richard}, {\em Stochastic
  quasi-newton langevin monte carlo}, 2016.

\bibitem{vollmer2016exploration}
{\sc S.~J. Vollmer, K.~C. Zygalakis, and Y.~W. Teh}, {\em Exploration of the
  (non-) asymptotic bias and variance of stochastic gradient langevin
  dynamics}, The Journal of Machine Learning Research, 17 (2016),
  pp.~5504--5548.

\bibitem{wang_multiphase}
{\sc Y.~Wang and G.~Lin}, {\em Efficient deep learning techniques for
  multiphase flow simulation in heterogeneous porousc media.}, Journal of
  Computational Physics, 401 (2020), p.~108968.

\bibitem{sgld}
{\sc M.~Welling and Y.~W. Teh}, {\em Bayesian learning via stochastic gradient
  langevin dynamicsn}, In Proceedings of the 28th international conference on
  machine learning (ICML-11),  (2011), pp.~681--688.

\bibitem{xifara2014langevin}
{\sc T.~Xifara, C.~Sherlock, S.~Livingstone, S.~Byrne, and M.~Girolami}, {\em
  Langevin diffusions and the metropolis-adjusted langevin algorithm},
  Statistics \& Probability Letters, 91 (2014), pp.~14--19.

\bibitem{zhang2011quasi}
{\sc Y.~Zhang and C.~A. Sutton}, {\em Quasi-newton methods for markov chain
  monte carlo}, in Advances in Neural Information Processing Systems, 2011,
  pp.~2393--2401.

\end{thebibliography}

\end{document}